\documentclass[11pt]{article}

\usepackage[english]{babel}
\usepackage[letterpaper,top=2cm,bottom=2cm,left=3cm,right=3cm,marginparwidth=1.75cm]{geometry}

\usepackage{amsmath,amssymb,amsfonts,mathrsfs}
\usepackage{amsthm}
\usepackage{bm}
\usepackage{graphicx}
\usepackage{enumerate}
\usepackage{hyperref}
\usepackage{mdframed}
\usepackage{xcolor}
\usepackage[titletoc,title]{appendix}

\numberwithin{equation}{section}

\newtheorem{theorem}{Theorem}[section]
\newtheorem{proposition}[theorem]{Proposition}
\newtheorem{lemma}[theorem]{Lemma}

\theoremstyle{definition}

\newtheorem{remark}[theorem]{Remark}

{\begin{mdframed}[linewidth=0.8pt]\noindent\textbf{#1.}}%
{\end{mdframed}}

\title{A Time-Domain Pressure–Interface Model for Gas Bubble Dynamics with Surface Tension: Well-Posedness, Classical Limits, and Resonance Branches}

\author{Long Li \thanks{{RICAM, Austrian Academy of Sciences, A-4040, Linz, Austria (long.li@ricam.oeaw.ac.at)}} \; and Mourad Sini \thanks{{RICAM, Austrian Academy of Sciences, A-4040, Linz, Austria (mourad.sini@oeaw.ac.at)}}}

\begin{document}

\maketitle
\begin{abstract}
\noindent We derive and analyze a time-domain pressure--interface evolution model for a
compressible gas bubble in a compressible liquid with surface tension. Starting
from the nonlinear two-phase Euler free-boundary problem and linearizing about
a spherical Young--Laplace equilibrium, we obtain a coupled bulk--surface
hyperbolic system for the liquid pressure $p_\ell$, the gas pressure $p_g$,
and the normal displacement $\eta$ of the interface:
\begin{align*}
\frac{1}{K_\alpha}\partial_t^2 p_\alpha
-\nabla\cdot\left(\frac{1}{\rho_\alpha}\nabla p_\alpha\right)=0,
\qquad \alpha\in\{\ell,g\},
\end{align*}
together with
\begin{align*}
\frac{1}{\rho_\ell}\partial_n p_\ell
=
\frac{1}{\rho_g}\partial_n p_g
=
-\partial_t^2\eta,
\qquad
p_\ell-p_g
=
\sigma\left(\Delta_{S^{(0)}}+\frac{2}{R_0^2}\right)\eta
\quad\text{on }S^{(0)} .
\end{align*}

\noindent The time-domain analysis is complicated by the fact that the surface-tension
quadratic form is indefinite: the \(Y_0^0\) component is the
volume-changing breathing mode, the \(Y_1^m (m = \{-1,0,1\})\) components are neutral
translations, and only the higher spherical harmonics give a coercive
shape-mode energy. We exploit this decomposition. The coercive sector is
treated by a \(C^0\)-semigroup argument, while the breathing mode is analyzed
separately by Fourier--Laplace methods. This yields well-posedness for
admissible finite-energy data and classical solutions under the natural
compatibility conditions.
\medskip

\noindent We then justify two limiting descriptions with quantitative error estimates.
In an acoustic quasi-static regime, the model reduces to the linearized
Rayleigh--Plesset equation for the breathing mode and to the linearized
Rayleigh--Lamb equations for the shape modes. In a different regime, it reduces
to the frozen-interface acoustic transmission model. Finally, a
frequency-domain analysis identifies the corresponding Minnaert,
Rayleigh--Lamb, and Fabry--P\'erot-type resonance mechanisms as different
components and limits of the same pressure--interface formulation.
\end{abstract}

\section{Introduction}
\subsection{Motivation}
Gas bubbles in liquids are fundamental objects in acoustics, fluid mechanics,
ultrasound imaging, bubbly media, and acoustic metamaterials \cite{AinslieLeighton2011,ADHO, Leighton1994,Leighton2004Acoustics,vanWijngaarden1972}. Their importance
comes from the fact that even a single bubble exhibits several distinct
resonance mechanisms. The Minnaert resonance is a volume-changing breathing
oscillation driven by gas compressibility and liquid added mass \cite{Minnaert}. The
Rayleigh--Lamb modes are zero-volume interfacial shape oscillations for which
surface tension is the leading restoring mechanism \cite{Lamb1932,Prosperetti1980FreeOsc}. At higher frequencies, the
gas domain can also support interior acoustic, i.e., Fabry--P\'erot-type modes \cite{P_1}.
These three mechanisms are often studied using different reduced models and
different asymptotic regimes:

\begin{enumerate}
\item One route starts from the free-boundary motion of a bubble
\cite{BV_00,CTW_13,LW_23,LW_25,SW_11}.
In the radial setting this viewpoint leads to Rayleigh--Plesset-type
equations
\cite{CommanderProsperetti1989, PlessetProsperetti1977,Prosperetti1982GenRP,Prosperetti1974Nonlinear},
whereas nonspherical perturbations lead to Rayleigh--Lamb shape oscillations
\cite{SW_11}. These models capture the
motion of the interface and the role of surface tension in shape dynamics.
However, in many reduced descriptions the gas dynamics is closed by an
isobaric or spatially uniform thermodynamic law, rather than by retaining the
full acoustic wave field inside the gas bubble.

\item A different route is provided by high-contrast acoustic transmission and scattering models
on fixed interfaces \cite{CaflischMiksisPapanicolaouTing1985,CaflischMiksisPapanicolaouTing1986,Miksis1991}. These models retain wave propagation in the gas and liquid domains and have been very successful in the analysis of Minnaert-type
subwavelength resonances \cite{Ammari-1} and interior acoustic modes, Fabry-P\'erot type (high-frequency) resonances, \cite{LiSiniFabryPerot}. However, by freezing the
bubble boundary, they suppress the interface displacement as a dynamical
unknown and therefore do not directly describe the free-boundary shape
dynamics generated by surface tension.
\end{enumerate}

\noindent One of the purposes of this paper is to bridge these two viewpoints at the linearized
time-domain level. We derive a pressure--interface formulation in which the
liquid and gas pressures satisfy acoustic wave equations in the bulk, while
the interface displacement is retained as an independent dynamical unknown.
Thus the model keeps both ingredients that are usually separated: the acoustic
wave field inside the bubble and the motion of the free boundary. As shown later, such models, which are closer to reality as compared to the frozen one, generate, in particular, richer resonance mechanisms including material contrasting based ones (as Minnaert and Fabry-P\'erot) and shape based ones (i.e. Rayleigh–Lamb ones).
\medskip

\noindent Another motivation for studying the present pressure--interface model is
 driven by its potential to provide a rigorous microscopic foundation for
several recent developments in imaging, wave control, and material science
based on resonant bubbles. In a sequence of works, resonant bubbles have been
successfully exploited as contrast agents for inverse problems, enabling the
recovery of material parameters through their subwavelength resonances and the
design of imaging modalities with enhanced sensitivity \cite{DabrowskiGhandricheSini2021,  Senapati-Sini}. Parallel developments
have established point-interaction models for
collections of resonating bubbles, as well as dispersive effective media and
metasurfaces generated by suitably distributed bubble assemblies \cite{HDAM-1, Ammari-2, Mukherjee-Sini-1}. More
recently, these ideas have found
applications in resonator-assisted control and acoustic cavitation \cite{Mukherjee-Sini-3}. A common feature of these contributions is that each bubble is represented by
an effective resonant degree of freedom, typically associated with its
dominant Minnaert mode. While this approximation has proved remarkably
successful, it does not explicitly account for the detailed coupling between
the acoustic pressure and the evolution of the fluid interface. The present
pressure--surface formulation fills this gap by retaining the full interface
dynamics together with the acoustic fields in both phases. As such, it
provides a natural PDE framework from which reduced Foldy--Lax systems,
effective medium theories, and dispersive models may be derived while
systematically incorporating capillary effects, higher-order surface modes,
and cavity resonances. This perspective opens several promising directions. In imaging with
resonant contrast agents, it suggests the possibility of exploiting not only
the Minnaert resonance but also shape and cavity resonances to enhance
spectral resolution and parameter reconstruction. In resonator-based control
and stabilization of wave equations, the additional interface degrees of
freedom may furnish new mechanisms for frequency-selective actuation and
energy transfer. Finally, in acoustic metamaterials and metasurfaces, the
present model provides a mathematically consistent route toward effective
constitutive laws in which dispersion arises from the combined influence of
breathing, capillary, and Fabry--P\'erot resonances rather than from a single
effective oscillator.
\medskip

\noindent Consequently, the present work should be viewed not only as a unified theory
of linear bubble dynamics, but also as a foundational step toward extending
recent mathematical developments on resonant imaging, effective media, and resonator-assisted wave manipulation and control to
settings where surface tension and interface dynamics play a fundamental role.

\bigskip

\subsection{The results}
\noindent\textbf{Mathematical setting.}
Let
\begin{align} \label{eq:equilibrium-domains} 
\Omega_\ell^{(0)}=:\mathbb R^3\setminus\overline{B^{(0)}},
\qquad
\Omega_g^{(0)}=:B^{(0)},
\qquad
S^{(0)}=:\partial B^{(0)}.
\end{align}
Here $B^{(0)}$ denotes the shperical equilibrium configuration of the gas
bubble, with centre at origin and radius
$R_{0}>0$, while $\Omega_\ell^{(0)}$ is the surrounding liquid domain. The liquid
has equilibrium density $\rho_\ell$ and bulk modulus $K_\ell$, and the gas
inside the bubble has equilibrium density $\rho_{g}$ and bulk
modulus $K_{g}$. The choice of a spherical equilibrium is not merely a simplifying assumption. For a static bubble with constant interior and exterior equilibrium pressures and constant surface tension, the Young-Laplace law gives
\begin{align*}
p_{g,0}-p_{\ell,0}=\sigma \kappa_0 .
\end{align*}
Thus the mean curvature $\kappa_0$ must be constant. In the closed embedded connected setting, this selects a round sphere. Hence the spherical equilibrium is not an additional geometric assumption, but a consequence of the modelling assumptions and the equilibrium equations themselves. It is the reference configuration about which we linearize.
\medskip

\noindent The central object of this work is the following linearized pressure--interface
system. The unknown $p$ denotes the acoustic pressure perturbation, restricted
to the liquid region $\Omega_\ell^{(0)}$ and to the gas bubble region $\Omega_g^{(0)}$.
The scalar function $\eta$ denotes the normal displacement of the
bubble interface from its equilibrium position. The pressure field satisfies
\begin{align}
\frac{1}{K_{\alpha}}p_{tt}
-\nabla\cdot\Big(\frac{1}{\rho_{\alpha}}\nabla p\Big) &= 0, 
\qquad \text{in }\Omega_\alpha^{(0)}\times(0,\infty), \qquad \alpha \in \{l,g\}.
\label{eq:bulk-wave-liquid-final}
\end{align}
The coupling on each equilibrium interface
\(S^{(0)}\) is given by
\begin{align}
\frac1{\rho_\ell}\partial_n p_\ell\big|
=
\frac1{\rho_{g}}\partial_n p_g\big|
&= -\partial^2_{t}\eta
&& \text{on }S^{(0)}\times(0,\infty),
\\
p_\ell - p_g
& = \sigma
\left(\Delta_{S^{(0)}}+\frac{2}{R_{0}^2}\right)\eta
&& \text{on}\;S^{(0)}\times(0,\infty).
\label{eq:jump-interface-final}
\end{align}
The first interface condition is the linearized kinematic condition, expressing
the continuity of the normal acceleration across the interface. The second one
is the linearized Young--Laplace law, where $\sigma>0$ is the surface tension
coefficient and $\Delta_{S^{(0)}}$ denotes the Laplace--Beltrami operator on $S^{(0)}$.

\medskip
\noindent\textbf{Main findings.}
The paper is organized around four quantitative results.
\begin{enumerate}[(1)]
\item 
The first result is the derivation of the pressure--interface model (\ref{eq:bulk-wave-liquid-final})-(\ref{eq:jump-interface-final}) from the
two-phase Euler free-boundary system with surface tension. The derivation is
performed around a spherical Young--Laplace equilibrium and keeps the full
normal displacement of the interface.

\item 
The second result is the time-domain well-posedness theory. After decomposing
the interface displacement into the breathing mode, the translation modes, and
the coercive shape-mode sector, we prove well-posedness of the coercive part by
a semigroup method (Theorem \ref{thm:wellposed-A}) and treat the breathing mode by Fourier--Laplace analysis (Theorem \ref{thm:wellposed-B}).
The combination yields finite-energy solutions and classical solutions under
the natural compatibility assumptions. Precisley, the surface operator
\begin{align*}
\Delta_{S^{(0)}}+\frac{2}{R_0^2}
\end{align*}
admits the orthogonal decomposition
\begin{align*}
L^2(S^{(0)})=\mathcal Y_0\oplus\mathcal Y_1\oplus\mathcal Y_{\ge2}.
\end{align*}
Here $\mathcal Y_0$ is spanned by $Y_0^0$, $\mathcal Y_1$ is spanned by the three
$l=1$ spherical harmonics, and $\mathcal Y_{\ge2}$ is the closed span of the higher modes.
These components correspond respectively to the breathing mode, rigid
translations, and higher-order shape oscillations. On $\mathcal Y_{\ge2}$ the associated energy
is coercive and the evolution generates a strongly continuous semigroup on the
natural energy space. The breathing mode is analyzed separately by
Fourier--Laplace methods, yielding existence, uniqueness and regularity for
the complete coupled pressure--interface problem.

\item 
The third result is the justification of two classical reductions. In the
acoustic quasi-static regime, the pressure--interface system reduces to the
linear Rayleigh--Plesset equation for the \(Y_0^0\) component and to the
Rayleigh--Lamb equations for the higher modes (Theorem \ref{th:reduct_1}). In a separate regime, the
moving-interface model reduces to the frozen-interface acoustic transmission
problem (Theorem \ref{prop:sigma-to-zero}). Both reductions are accompanied by error estimates:

\begin{enumerate}
\item[(i)]
For every spherical harmonic degree $l\ge2$, the dominant interface dynamics
are governed by the Rayleigh--Lamb equation
\[
\left(
\frac{\rho_\ell}{l+1}
+
\frac{\rho_g}{l}
\right)
\ddot\eta_l
-
\frac{\sigma(l+2)(l-1)}{R_0^3}\eta_l
=0,
\]
whose oscillation frequency is
\[
\omega_l
=
\sqrt{
\frac{(l+2)(l-1)\sigma}
{\left(
\frac{\rho_\ell}{l+1}
+
\frac{\rho_g}{l}
\right)R_0^3}
}.
\]

\item[(ii)]
The monopole component reduces to the linearized Rayleigh--Plesset equation,
recovering the classical Minnaert breathing dynamics. A key point is that the spatially uniform
gas-pressure law is derived, rather than imposed: it follows from a
well-prepared constant interior acoustic state together with an acoustic
quasi-static regime. We justify this reduction by quantitative finite-time
error estimates, thereby identifying the assumptions under which the linearized
Rayleigh--Plesset closure is valid within the pressure--interface wave model.

\item[(iii)]
As the surface tension tends to zero, the solutions converge to those of the
standard frozen-interface acoustic transmission problem, with quantitative
error estimates established in Sobolev spaces.
\end{enumerate}

\item 
Finally, we include a frequency-domain analysis of the associated scattering
problem (Theorem \ref{prop:resonance}). This identifies the resonance mechanisms corresponding to the
time-domain model: the Minnaert branch from the volume-changing sector, the
Rayleigh--Lamb branch from the shape-mode sector, and  Fabry--P\'erot-type branches from finite-wavelength gas oscillations: 

\begin{enumerate}

\item[(i)] \emph{Minnaert branch.}
The monopole ($l=0$) component gives rise to the breathing resonance. Its
resonance frequencies are perturbations of the classical Minnaert frequency
and satisfy
\begin{align*}
\omega_{M}^{\pm}
=
\pm
\sqrt{
\frac{3K_g}{\rho_\ell R_0^{2}}
-
\frac{2\sigma}{\rho_\ell R_0^{3}}
}
+
o(1).
\end{align*}
This leading-order behavior is governed by the effective radial coefficient
\begin{align*}
\Omega_M^2
:=
\frac{3K_g}{\rho_\ell R_0^{2}}
-
\frac{2\sigma}{\rho_\ell R_0^{3}} .
\end{align*}
When $\Omega_M^2>0$, the Minnaert poles form an oscillatory pair near
$\pm\Omega_M$. If instead $
3K_gR_0<2\sigma,$ then $\Omega_M^2<0$, and the leading-order pair becomes purely imaginary.
Although this regime is not the usual stable physical regime, it is
mathematically significant: with the \(e^{-i\omega t}\) convention, one of
the corresponding poles lies in the upper half-plane.

\item[(ii)] \emph{Rayleigh--Lamb branch.}
For every spherical harmonic degree $l\ge2$, the interface shape modes possess
a pair of resonances satisfying
\[
\omega_{L}^{(l),\pm}
=
\pm
\sqrt{
\frac{(l+2)(l-1)\sigma}
{
\left(
\frac{\rho_\ell}{l+1}
+
\frac{\rho_g}{l}
\right)
R_0^{3}
}
}
+
o(1),
\]
which coincide, at leading order, with the classical Rayleigh--Lamb
frequencies associated with capillary oscillations of degree $l$.

\item[(iii)] \emph{Fabry--P\'erot branch.}
Besides the low-frequency Minnaert and capillary branches, the coupled model
possesses an infinite family of interior acoustic resonances. These
frequencies correspond to standing waves inside the gas domain and satisfy the
interior Helmholtz problem
\[
\Delta u + \lambda^2 u =0
\qquad\text{in }B_{R_0},
\]
together with the limiting boundary condition induced by the interface
dynamics. In the asymptotic regime in which the interface becomes effectively
rigid (equivalently, when compressibility effects dominate the interface
motion), the corresponding resonance frequencies converge to the acoustic
eigenfrequencies of the cavity,
\[
\omega_{F,n}^{\pm}
=
\pm \lambda_n
+
o(1),
\qquad n=1,2,\ldots,
\]
where $\{\lambda_n\}$ denotes the Neumann spectrum of the ball
$B_{R_0}$. Thus the Fabry--P\'erot resonances are interpreted as the
high-frequency cavity branch of the same pressure--interface operator.

\end{enumerate}

Consequently, the proposed formulation unifies in a single mathematical model
three qualitatively different oscillatory mechanisms:
the low-frequency Minnaert breathing resonance, the capillary
Rayleigh--Lamb shape resonances, and the interior Fabry--P\'erot cavity
resonances. Their coexistence and interaction are recovered without changing
the governing equations, but only by analyzing different spectral sectors of
the same coupled pressure--interface system.
\medskip

\end{enumerate}

\noindent The novelty of the present work lies in the introduction and analysis of a unified
pressure--interface formulation that simultaneously retains the acoustic
fields in both phases and the full dynamics of the moving boundary. To the best of our knowledge, this is the first rigorous framework that reconciles the classical free-boundary descriptions of Rayleigh--Plesset and
Rayleigh--Lamb type with the pressure-based formulations underlying subwavelength Minnaert resonances and acoustic cavity modes. On the analytical side, the model gives rise to an indefinite interface operator whose spectral
structure combines a negative breathing mode, a translational kernel, and a coercive higher-order sector. The corresponding well-posedness theory requires a combination of semigroup methods and Fourier--Laplace techniques, leading to a unified treatment of these different regimes. Beyond its intrinsic mathematical interest, the proposed formulation provides a natural microscopic foundation for effective theories of resonant bubble assemblies. In particular, it creates the possibility of systematically incorporating surface-tension effects into Foldy--Lax approximations, homogenized media, and resonator-based models for imaging and wave control, thereby establishing a bridge between free-boundary fluid mechanics and
modern theories of resonant wave propagation.

\bigskip
\noindent Extending the analysis above to non-spherical equilibria would require introducing additional physical mechanisms that modify the equilibrium balance and, after linearization, affect not only the curvature operator but the entire coupled pressure–interface system.

\medskip
\noindent\textbf{Structure of the paper.}
The paper is organized as follows. Section~\ref{sec:model-linearization} introduces the nonlinear two-phase model, derives the linearized pressure--interface formulation, and introduces the
spherical harmonic notation. Section~\ref{sec:wellposedness} proves the
well-posedness of the linearized system: the \(Y_0^0\) sector is treated by a
weighted Fourier--Laplace argument, while the orthogonal component is handled
by an energy-based semigroup method. Section~\ref{sec:reduced-models} discusses asymptotic reductions to the frozen-interface and linearized Rayleigh--Plesset-type models. Section~\ref{sec:resonance-connection} analyzes the resonance regimes and recovers the Minnaert, Lamb and Fabry--P\'erot branches. Finally, in the appendix we collect the proofs of several technical lemmas and auxiliary estimates used in the proofs of the main theorems.

\section{Model derivation and linearization}
\label{sec:model-linearization}

\subsection{Nonlinear two-phase model in a heterogeneous background}

\paragraph{Geometry and phases.}

Let  $\mathcal B(t)\subset\mathbb R^3$ denote the domain
occupied by the gas bubble at time $t$, and let
$\Gamma(t):=\partial \mathcal B(t)$ be its interface with the surrounding liquid.
The surrounding liquid occupies the unbounded exterior domain
$
\Omega_\ell(t):=\mathbb{R}^3\setminus
\overline{\mathcal B(t)}.
\notag
$
On the interface $\Gamma(t)$, we denote by $\nu(t)$ the unit normal
pointing from the gas bubble into the liquid.

The liquid variables are denoted by
$\varrho_\ell(x,t),\; v_\ell(x,t)$ and $P_\ell(x,t)
$
representing density, velocity, and pressure, respectively. Inside the
bubble, the gas variables are denoted by
$
\varrho_{g}(x,t),\; v_{g}(x,t),$ and $P_{g}(x,t).
$

\paragraph{Governing equations in each phase.}

We consider inviscid, compressible, barotropic flow in each phase.

In the liquid domain $\Omega_\ell(t)$, mass and momentum conservation give
\begin{align}
  \partial_t \varrho_\ell + \nabla\cdot(\varrho_\ell v_\ell) &= 0,
    && x \in \Omega_\ell(t),\ t>0,
    \label{eq:mass-liquid}\\
  \varrho_\ell\Bigl(\partial_t v_\ell + (v_\ell\cdot\nabla) v_\ell \Bigr)
   + \nabla P_\ell &= 0,
    && x \in \Omega_\ell(t),\ t>0.
\end{align}
The liquid is assumed to be barotropic in the sense that its pressure is
determined by the local density through the equation of state
$P_\ell= {P_\ell(\varrho_\ell)}.$ 

Inside the bubble $B(t)$, mass and momentum conservation give
\begin{align}
  \partial_t \varrho_{g} + \nabla\cdot(\varrho_{g} v_{g})
    &= 0,
    && x \in \mathcal B(t),\ t>0,\\
  \varrho_{g}\Bigl(\partial_t v_{g} + (v_{g}\cdot\nabla) v_{g} \Bigr)
   + \nabla P_{g}
    &= 0,
    && x \in \mathcal B(t),\ t>0.
\end{align}
The gas in the bubble is also assumed to be barotropic, with
\begin{align*}
  P_{g} = P_{g}(\varrho_{g}).
\end{align*}

\paragraph{Interface conditions.}

On the moving interface $\Gamma(t)$, we impose the kinematic condition
\begin{align} \label{eq:kinematic-full}
  V_{\Gamma}=v_\ell\cdot \nu=v_{g}\cdot \nu
  \qquad\text{on }\Gamma(t).
\end{align}
Here, $V_{\Gamma}$ denotes the normal velocity of the interface in the
direction $\nu$.

The dynamic condition is the Young--Laplace law. With the convention, we define the mean curvature of $\Gamma(t)$ by
\begin{align*}
\kappa:=\nabla_{\Gamma(t)}\cdot \nu,
\end{align*}
and then we write
\begin{align} \label{eq:dynamic-full}
  P_{g}-P_\ell= {\sigma\kappa}
  \qquad\text{on }\Gamma(t).
\end{align}

\paragraph{Equilibrium state and perturbation variables.}

We consider small perturbations around a homogeneous equilibrium state. In the
liquid phase, the equilibrium velocity is zero, and the equilibrium density
and pressure are constants
\begin{align*}
  v_\ell^{(0)}=0,\qquad
  \rho_\ell>0,\qquad
  p_\ell = {P_\ell(\rho_\ell)}.
\end{align*}

Similarly, in the gas bubble, the equilibrium velocity is zero, and the
equilibrium density and pressure are constants
\begin{align*}
  v_{g}^{(0)}=0,\qquad
  \rho_{g}>0,\qquad
  p_{g}= {P_{g}(\rho_{g})}.
\end{align*}
Moreover, we assume the existence of a spherical equilibrium configuration, with the $j$- th equilibrium bubble domain $B_j$ and the equilibrium liquid
domain $\Omega_\ell^{(0)}$, defined in \eqref{eq:equilibrium-domains}. This
assumption is natural in the present setting: if the equilibrium pressures are
constant in each phase, then the Young--Laplace law gives
\begin{align*}
  p_{g} - p_\ell
  =
  {\sigma \kappa^{\mathrm{eq}}}
  \qquad \text{on } \partial \mathcal B .
\end{align*}
Thus the equilibrium interface has constant mean curvature. In particular, a
sphere provides the canonical equilibrium shape, for which
\begin{align*}
  \kappa^{\mathrm{eq}}=\frac{2}{R_{0}}
\qquad\mathrm{and}\qquad 
  p_{g} - p_\ell
  =  {\frac{2\sigma}{R_{0}}}.
\end{align*}

\begin{remark}[Example of a gas equation of state]
A standard example is the polytropic equation of state
\begin{align}
  P_{g} = A\varrho_{g}^{\gamma},\qquad \gamma>1,
  \label{eq:polytropic-EOS}
\end{align}
where $\gamma$ is the effective polytropic exponent. The constant
$A>0$ is chosen so that \eqref{eq:polytropic-EOS} matches the equilibrium
state $(\rho_{g},p_{g})$, i.e.
$A = {p_{g}}/\rho^\gamma_{g}$.
Such polytropic laws are routinely used in the modelling of acoustically
driven bubbles, with $\gamma$ between the isothermal value $\gamma=1$ and
the adiabatic value $\gamma=c_p/c_g$, depending on the frequency regime and
the thermal properties of the gas and surrounding liquid
\cite{Brennen1995, PlessetProsperetti1977}.
\end{remark}

\subsection{Linearization of the bulk equations and moving interface}

We now state the linearization in a condensed proposition-style form and keep only the structural first-order terms needed later in the paper.

We introduce a small parameter $0<\varepsilon\ll 1$ and perturb the equilibrium on the fixed reference geometry. Routine Taylor expansions are omitted from the main line of the argument; only the resulting $O(\varepsilon)$ identities are recorded.
For the bulk fields, we use the ansatz
\begin{align}
v_\alpha(x,t) & = \varepsilon u_\alpha(x,t) + O(\varepsilon^2),
\label{eq:exp-vl}\\
\varrho_\alpha(x,t) & = \rho_{\alpha} + \varepsilon \rho^{(1)}_{\alpha}(x,t) + O(\varepsilon^2),\\
P_\alpha(x,t) &= p_{\alpha} + \varepsilon p^{(1)}_{\alpha}(x,t) + O(\varepsilon^2), \qquad \alpha \in \{\ell,g\}.
\label{eq:exp-pl}
\end{align}

The moving interface can be represented as a graph over the equilibrium sphere:
\begin{align} \label{eq:normal-displacement}
X(n,t)
=
\varepsilon \xi(t)
+
\bigl(R_{0}+\varepsilon h(n,t)\bigr) n.
\end{align}
where $n$ is the outward unit normal direction to the equilibrium
sphere, $\xi(t)\in\mathbb R^3$ represents the translational displacement of the
bubble centre, while $h$ represents the radial deformation of
the interface. Since the degree-one spherical harmonics correspond to infinitesimal
translations of the sphere, we separate translations from shape
deformations by imposing the gauge condition
\begin{align} \label{eq:gauge}
  \left\langle h, Y^{l}_{1,0} \right\rangle_{L^2\left(S^{(0)}\right)}=0, \quad l \in \{-1,0,1\}.
\end{align}
Here, 
\begin{align*}
  Y_{1,0}^l(x)
  :=
  Y_1^l\left(\frac{x}{R_{0}}\right),
  \qquad x\in S^{(0)},
\end{align*}
where $Y_1^l$ denotes the usual degree-one spherical harmonic on $\mathbb S^2$. Throughout this paper, for a Hilbert space $X$, we denote its
inner product by $\langle \cdot,\cdot\rangle_X$.
The above condition \eqref{eq:gauge} removes the translational component from $h$: translations
are represented by $\xi(t)$, whereas $h$ contains only
shape-deformation modes. For later use, it is convenient to introduce the
combined scalar normal displacement
\begin{align*}
  \eta(n,t)
  :=
  \xi(t)\cdot n + h( n,t).
\end{align*}
The term $\xi(t)\cdot n$ is a degree-one spherical harmonic and
represents the normal component of the centre translation. The gauge condition \eqref{eq:gauge}
therefore separates translations from genuine shape modes and makes the decomposition of $\eta$ unique.

\begin{proposition}[First-order linearization about the equilibrium]
\label{prop:first-order-linearization}
Under the expansion \eqref{eq:exp-vl}, and the representation \eqref{eq:normal-displacement}, the nonlinear two-phase system \eqref{eq:mass-liquid}--\eqref{eq:dynamic-full} reduces at order $O(\varepsilon)$ to the following equations on the equilibrium domains.

\smallskip
\noindent\emph{(i) Bulk equations }
The first-order bulk fields satisfy
\begin{align}
  \rho_{\alpha}\,\partial_t u_\alpha(x,t) + \nabla p^{(1)}_{\alpha}(x,t) &= 0,
  && x\in\Omega_\alpha^{(0)},\ t>0, \notag \\
  \frac{1}{K_\alpha}\,\partial_t p^{(1)}_{\alpha}(x,t) + \nabla\cdot u_\alpha(x,t) &= 0,
  && x\in\Omega_\alpha^{(0)},\ t>0\notag 
\end{align}
with 
\begin{align*}
 K_{\alpha} = \rho_{\alpha} P'_{\alpha}(\rho_{\alpha}), \qquad \alpha \in \{\ell,g\}.
\end{align*}


\smallskip
\noindent\emph{(iii) Interface conditions on $S^{(0)}$.}
The kinematic condition becomes
\begin{align}
  \partial_t \eta(\theta,\varphi,t)
   = u_\ell(x,t)\cdot n(x)
   = u_{g}(x,t)\cdot n(x),
  \quad x\in S^{(0)},\ t>0.
  \label{eq:lin-kinematic}
\end{align}
The dynamic condition reduces to
\begin{align}
  p^{(1)}_{\ell}(x,t) - p^{(1)}_{g}(x,t)
   = \sigma \left(\Delta_{S^{(0)}}+\frac{2}{R_{0}^2}\right)\eta(x,t),
  \quad x\in S^{(0)},\ t>0.
  \label{eq:lin-dynamic}
\end{align}
\end{proposition}

\begin{proof}[Derivation sketch]
The bulk equations follow by inserting \eqref{eq:exp-vl}--\eqref{eq:exp-pl} into the momentum and pressure-continuity formulations and collecting the $O(\varepsilon)$ terms. For the interface, differentiating \eqref{eq:normal-displacement} in time gives 
\begin{align*}
\partial_t X=\varepsilon \partial_t\eta\,n + O(\varepsilon^2),
\end{align*}
which inserted into \eqref{eq:kinematic-full} yields \eqref{eq:lin-kinematic}. The dynamic condition follows from the first-order expansion of \eqref{eq:dynamic-full} around the Young--Laplace equilibrium law. The curvature variation formula is standard shape calculus; see, for instance, \cite{HenrotPierre2018, LamboleyShapeIntro}. This gives \eqref{eq:lin-dynamic}.

In the radial-mode truncation, $\eta(\theta,\varphi,t)\equiv \eta(t)$ and hence
\begin{align}
  p^{(1)}_{\ell}(x,t) - p^{(1)}_{g}(x,t)
   = \frac{2\sigma}{R_{0}^2}\eta(t),
   \quad x\in S^{(0)},\ t>0.
  \label{eq:breathing-dynamic}
\end{align}
Clearly, \eqref{eq:breathing-dynamic} is its restriction to the $\ell=0$ spherical harmonic sector.
\end{proof}

Introducing a single pressure field $p(x,t)$ defined piecewise by
\begin{align*}
  p(x,t) :=
\begin{cases}
p^{(1)}_{\ell}(x,t), & x\in\Omega_\ell^{(0)},\\[0.2em]
p^{(1)}_{g}(x,t), & x\in \Omega_g^{(0)}.
\end{cases}
\end{align*}
From now on, we use $p_\ell$ and $q_\ell$
to denote the restrictions of $p$ and $q$ to the liquid domain, while $p_{g}$ and $q_{g}$ denote their restrictions to the gas bubble domain.
Using the interface relations obtained in Proposition \ref{prop:first-order-linearization}, we derive an equivalent fixed-domain pressure--interface formulation of the linearized
moving-interface system. In this formulation, the velocity fields and density perturbations are eliminated, and the remaining unknowns are the pressure field $p$ and the
interface displacement $\eta$. The resulting system is
\eqref{eq:bulk-wave-liquid-final}--\eqref{eq:jump-interface-final}.
For later use, we introduce the velocity variables
\begin{align*}
q_\alpha:=\partial_t p_\alpha,\qquad \zeta:=\partial_t\eta,
\end{align*}
and prescribe initial data
\begin{align*}
p_\alpha(0) = p_{\alpha,0},\qquad
q_\alpha(0) = q_{\alpha,0},\qquad
\eta(0)=\eta_0,\qquad
\zeta(0)=\zeta_0 .
\end{align*}
The precise regularity and compatibility assumptions on these data will be specified in the well-posedness analysis. Since the reference interface is spherical, we shall repeatedly use the spherical harmonic decomposition introduced in the next subsection.

\subsection{Spherical harmonic decomposition}
\label{subsec:spherical-harmonics}

Since the reference interface is spherical, the linearized dynamics is
naturally decomposed by spherical harmonics. We write $x=r\hat x$, with
$r>0$ and $\hat x\in\mathbb S^2$, and denote by $Y_l^m$ the spherical
harmonics on $\mathbb S^2$. We set
\begin{align*}
\mathcal Y_0:=\operatorname{span}\{Y_0^0\},\qquad
\mathcal Y_1:=\operatorname{span}\{Y_1^{-1},Y_1^0,Y_1^1\},
\end{align*}
and
\begin{align*}
\mathcal Y_{\ge2}:=
\overline{\operatorname{span}\{Y_l^m:\ l\ge2,\ -l\le m\le l\}}^{\,L^2(S^{(0)})}.
\end{align*}
Thus
\begin{align*}
L^2(S^{(0)})=\mathcal Y_0\oplus\mathcal Y_1\oplus\mathcal Y_{\ge2}.
\end{align*}
Let $\Pi_0$, $\Pi_1$, and $\Pi_{\ge2}$ denote the corresponding
$L^2(S^{(0)})$-orthogonal projections, and set
\begin{align*}
\Pi_{\ge1}:=I-\Pi_0=\Pi_1+\Pi_{\ge2}.
\end{align*}
For the pressure variables we use the decomposition
\begin{align*}
p_\alpha(r,\hat x,t) =
p_\alpha^{(0)}(r,t)Y_0^0(\hat x)+p_\alpha^\perp(r,\hat x,t),
\qquad
\Pi_0 p_\alpha^\perp=0,
\quad  \alpha\in\{\ell,g\},
\end{align*}
and similarly
\begin{align*}
\eta(\hat x,t) =
\eta^{(0)}(t)Y_0^0(\hat x)+\eta^\perp(\hat x,t),
\qquad
\Pi_0\eta^\perp=0 .
\end{align*}
The $Y_0^0$ sector is studied through the radial variables
\begin{align*}
v_\alpha(r,t):=rp_\alpha^{(0)}(r,t),
\qquad \alpha\in \{\ell,g\},
\end{align*}
whereas the orthogonal component is written as a first-order evolution system with
\begin{align*}
U=(p,q,\eta,\zeta).
\end{align*}
For traces on $S^{(0)}$, we use
\begin{align*}
[p]:=p_\ell|_{S^{(0)}}-p_g|_{S^{(0)}}.
\end{align*}
Whenever the flux continuity condition holds, we denote the common weighted normal flux by
\begin{align*}
\partial_n p:=\frac1{\rho_\ell}\partial_n p_\ell
=\frac1{\rho_g}\partial_n p_g.
\end{align*}

\paragraph{General notation.}
Throughout the paper, $H^s(\Omega)$ and $H^s(\Gamma)$ denote the standard Sobolev spaces on a smooth domain $\Omega$ and on the smooth surface $\Gamma$, respectively. Their norms are denoted by $\|\cdot\|_{H^s(\Omega)}$ and $\|\cdot\|_{H^s(\Gamma)}$. Surface differential operators, such as $\nabla_{S^{(0)}}$ and $\Delta_{S^{(0)}}$, are always taken with respect to the reference sphere $S^{(0)}$, unless otherwise specified. For product spaces, we use the natural product norms without further comment. The symbol $\mathbb I$ denotes the identity matrix. The letter $C$ denotes a generic positive constant which may change from line to line.

\section{Well-posedness}
\label{sec:wellposedness}

In this section, we prove the well-posedness of the linearized spherical
system using the modal decomposition introduced in
Subsection~\ref{subsec:spherical-harmonics}. Since the reference interface is
spherical, the $Y_0^0$ sector and its orthogonal complement are invariant
under the linearized dynamics. Hence the full problem splits into two
independent subsystems. The $Y_0^0$ sector is treated in
Subsection~\ref{sec:wellposedness_1} by a weighted Fourier--Laplace argument, while
the $\Pi_{\ge1}$-component is handled in
Subsection~\ref{sec:wellposdeness_2} by an energy-based semigroup method.
 
\subsection {\texorpdfstring {Weighted Fourier--Laplace well-posedness of the $Y_0^0$ sector}{}} \label{sec:wellposedness_1}

We first describe the projected $Y_0^0$ system. Let $v(r,t):=r p^{(0)}(r,t)$. 
Then the $Y_0^0$ sector system is
\begin{align}
    & \partial_{tt}v_g-c_g^2\partial_{rr}v_g=0,
        \quad 0<r<R_0, \label{eq:t_0} \\ 
    & \partial_{tt}v_\ell-c_\ell^2\partial_{rr}v_\ell=0,
        \quad r> R_0, \\ 
    & v_g(0,t)=0, \\ 
 &v_\ell(R_0,t)-v_g(R_0,t)=\frac{2\sigma}{R_0}\eta^{(0)}(t), \label{eq:t_1}\\
    &\dfrac1{\rho_g R_0^2}
    \left( R_0\partial_r v_g(R_0,t) - v_g(R_0,t)
    \right)
    =
    \dfrac1{\rho_\ell R_0^2}
    \left(
R_0\partial_rv_\ell(R_0,t)- v_\ell(R_0,t)
    \right)
    = -\frac{d^2}{dt^2}\eta^{(0)}(t)\label{eq:t_2}.
\end{align}
Here, 
\begin{align*}
    c_\alpha := \sqrt\frac{K_\alpha}{\rho_\alpha}.
\end{align*}

We now prove the well-posedness of the $Y_0^0$ system \eqref{eq:t_0}--\eqref{eq:t_2}. The argument is
based on a weighted Fourier--Laplace analysis. 
For a function $f$ defined on $t\ge0$, we use the one-sided Laplace
transform
\begin{align*}
    \widehat f(s):=\int_0^{+\infty} e^{-st}f(t)\,dt, \qquad s=\gamma+i\xi,\; \gamma>0.
\end{align*}
For an integer $m\ge0$, we write
\begin{align*}
    \|f\|_{H^m_\gamma(0,\infty)}^2
    :=
    \sum_{j=0}^m
     \int^{+\infty}_0 e^{-\gamma t}|\partial_t^j f(t)|^2 dt.
\end{align*}

We next introduce the boundary impedances and compliances associated with
the transformed one-dimensional gas and liquid wave equations. These
quantities encode the response of the boundary traces to the normal force
in the \(Y_0^0\) sector. For \(s=\gamma+i\xi\), set
\begin{align*}
    Z_g(s)
    := \frac1{\rho_g}
    \left(
        \frac{s}{R_0c_g}
    \coth\left(\frac{sR_0}{c_g}\right) -
    \frac1{R_0^2}
    \right),
\qquad
    Z_\ell(s)
    :=
    \frac1{\rho_\ell R_0^2}
    \left(
        1+\frac{sR_0}{c_\ell}
    \right).
\end{align*}
We denote the corresponding compliance functions by
\begin{align*}
 C_g(s):=\frac1{Z_g(s)},
    \qquad
C_\ell(s):=\frac1{Z_\ell(s)}.
\end{align*}
The scalar multiplier for the $Y_0^0$ boundary equation is defined by
\begin{align} \label{eq:multiplier}
    M(s)
    :=
    s\left(C_g(s)+C_\ell(s)\right)
    -
    \frac{2\sigma}{R_0s}.
\end{align}

The following lemma provides a lower bound for the multiplier $M$, which plays an important role in the proof of the well-posedness of the $Y_0^0$ sector. Its proof is postponed to the Appendix \ref{sec:A}.

\begin{lemma} \label{le:multiplier}
Let $M$ be given by \eqref{eq:multiplier} and let $R_0, \rho_{\ell,*},c_{\ell,*}, \sigma_*  >0$.
Assume that $\rho_\ell\ge\rho_{\ell,*}>0,\;
c_\ell\ge c_{\ell,*}>0$ and $0<\sigma\le \sigma_*$. Then, there exists $\gamma_*= \gamma_*(R_0, \rho_{\ell,*},c_{\ell,*},\sigma_*)>0$ 
such that, for every \(\gamma\ge\gamma_*\), there exists
$m_\gamma>0$ satisfying
\begin{align} \label{eq:estimate}
    |M(\gamma+i\xi)|\ge m_\gamma
    \qquad \text{for all}\; \xi\in\mathbb R.
\end{align}
Here, $m_\gamma$ is a positive constant.

\end{lemma}

We now introduce the class of admissible $Y_0^0$ initial data. For an
integer $m\ge1$, set
\begin{align*}
    I_g:=(0,R_0),
    \qquad
    I_\ell:=(R_0,\infty).
\end{align*}
For $\alpha\in\{g,\ell\}$, define
\begin{align*}
    V_\alpha^{[0]}:= r p^{(0)}_{\alpha,0},
    \qquad
    V_\alpha^{[1]}:= r q^{(0)}_{\alpha,0},
\end{align*}
and, recursively,
\begin{align*}
    V_\alpha^{[j+2]}
    :=
    c_\alpha^2\frac{d^2}{dr^2}V_\alpha^{[j]},
    \qquad j\ge0.
\end{align*}

An initial datum
\begin{align*}
    \mathcal I^{(0)}
    :=
    \left(
        V_\ell^{[0]}, V_g^{[0]}, V_\ell^{[1]},  V_g^{[1]}, \eta^{(0)}_0, \zeta^{(0)}_0 \right)
\end{align*}
is said to belong to \(\mathcal X^{(0)}_m\) if
\begin{align*}
    V_\alpha^{[0]}\in H^{m+1}(I_\alpha),
    \qquad
    V_\alpha^{[1]}\in H^m(I_\alpha),
    \qquad 
    \eta^{(0)}_0,\; \zeta^{(0)}_0\in\mathbb R,
\end{align*}
and the following compatibility conditions hold.
\begin{align}
    &V_g^{[j]}(0)=0, \notag\\
    &E_0:=\eta^{(0)},
    \qquad
    E_1:=\zeta^{(0)}, \qquad
    E_{j+2}
    :=
    -\mathcal B_g V_g^{[j]}
    =
    -\mathcal B_\ell V_\ell^{[j]}, \notag\\
    &V_\ell^{[j]}(R_0) - V_g^{[j]}(R_0)
    =
    \frac{2\sigma}{R_0}E_j,
    \qquad
    0\le j\le m. \label{eq:37}
\end{align}
Here, we write
\begin{align*}
    \mathcal B_\alpha \phi
    :=
    \frac1{\rho_\alpha R_0^2}
    \left(
        R_0\phi'(R_0)-\phi(R_0)
    \right).
\end{align*}
We measure the element in $\mathcal X^{(0)}_m$ by the norm 
\begin{align*}
\|\mathcal I\|_{\mathcal X^{(0)}_m} = \sum^{m}_{j=0}\left[|E_{j+1}|  + \sum_{\alpha \in \{\ell,g\}} \left\|\frac{d}{dr}V_\alpha^{[j]}\right\|_{L^2(I_\alpha)} + \sum_{\alpha \in \{\ell,g\}} \|V_\alpha^{[j+1]}\|_{L^2(I_\alpha)}\right]
\end{align*}

Now we state the well-posedness result in the $Y_0^0$ sector.

\begin{theorem}[Weighted well-posedness of the \(Y_0^0\) sector] \label{thm:wellposed-B}
Let $R_0, \rho_{\ell,*},  c_{\ell,*},  \rho_{g,*},  c_{g,*},\sigma_*>0$.
Assume that $\rho_\alpha \ge\rho_{\alpha,*}>0,\;
c_\alpha\ge c_{\alpha,*}>0$ for $\alpha \in \{\ell, g\}$ and $0<\sigma\le \sigma_*$. For every admissible initial datum
$\mathcal I^{(0)}\in\mathcal X^{(0)}_m$, the $Y_0^0$ system
\eqref{eq:t_0}--\eqref{eq:t_2} admits a unique weighted solution in the
following sense:
\begin{align} \label{eq:51}
    \frac{d}{dt}\eta^{(0)} \in H^m_\gamma(0,\infty),
    \qquad
    \eta^{(0)}\in H^{m}_\gamma(0,\infty),
    \qquad
    v_g(R_0,\cdot), \;\; v_\ell(R_0,\cdot)\in H^m_\gamma(0,\infty),
\end{align}
and, for $\alpha\in\{\ell,g\}$,
\begin{align} \label{eq:52}
    \partial_t^j v_\alpha
    \in L^2_\gamma(0,\infty;L^2(I_\alpha)),\qquad
    0\le j\le m.
\end{align}
Here, $\gamma>\widetilde \gamma_*$, where
$\widetilde\gamma_*=\gamma_*(R_0,\rho_{\ell,*},  c_{\ell,*},\rho_{g,*}, c_{g,*},\sigma_*)>0$ is a fixed threshold.

Moreover, if $m \ge 1$, we have that: for $\alpha\in\{g,\ell\}$,
\begin{align} \label{eq:53}
\partial_t^{j-1} v_\alpha
\in L^2_\gamma(0,\infty;H^1(I_\alpha)),\qquad 1\le j\le m.
\end{align}
\end{theorem}

\begin{proof}
Let $s = \gamma + i\xi$. We begin with looking for the solution of the following equations 
\begin{align}
 &Z_g(s)\widehat V_g(s)+\widehat H_g(s) =  -Z_\ell(s)\widehat V_\ell(s)+\widehat H_\ell(s) = (- s\hat \zeta(s) + \zeta^{(0)}), \label{eq:21} \\
 &\widehat V_\ell(s) - \widehat V_g(s) = \frac{2\sigma}{R_0} \frac{\hat \zeta(s) + \eta^{(0)}}{s}, \label{eq:22}
\end{align}
where 
\begin{align*}
\hat H_\alpha:= \frac 1{R_0 \rho_\alpha} \partial_r \hat v^f_\alpha\quad  \mathrm{for}\; \alpha \in \{\ell,g\}.
\end{align*}
Here,
\(\widehat v_g^{\,f}\) solves
\begin{align*}
    &s^2\widehat v_g^{\,f}
    -
    c_g^2\partial_{rr}\widehat v_g^{\,f}
    =
    s V_g^{[0]} +   V_g^{[1]},
    \qquad
    0<r<R_0,\\
    &\widehat v_g^{\,f}(s,0)=0,
    \qquad
    \widehat v_g^{\,f}(s,R_0)=0,
\end{align*}
and $\widehat v_\ell^{\,f}$ solves
\begin{align*}
& s^2\widehat v_\ell^{\,f}
    -
    c_\ell^2\partial_{rr}\widehat v_\ell^{\,f}
    =
    s  V_\ell^{[0]} +  V_\ell^{[1]},
    \qquad
    r>R_0,
& \widehat v_\ell^{\,f}(R_0,s)=0, \qquad \lim_{r\rightarrow \infty} v^f(R_0,s) = 0.
\end{align*}
Combining the above two equations gives the scalar equation below:
\begin{align*}
    M(s)\widehat \zeta(s)=\widehat{\mathcal H}(s),
\end{align*}
and
\begin{align*}
   \widehat{\mathcal H}(s)
    :=
    \left(C_g(s)+C_\ell(s)\right)\zeta^{(0)}
    -
    \frac{2\sigma\eta^{(0)}}{R_0s}
    +
    C_g(s)\widehat H_g(s)
    +
   C_\ell(s)\widehat H_\ell(s).
\end{align*}
It follows from  \eqref{eq:estimate} that 
\begin{align} \label{eq:20}
\|\widehat \zeta(\gamma+i\xi)\|_{L^2(\mathbb R)}
\le \frac1{m_\gamma} \|\widehat{\mathcal H}(\gamma+i\xi)\|_{L^2_\xi}, \qquad \gamma > \gamma^*,\; \xi \in \mathbb R.
\end{align}
Here, $\gamma^*$ is chosen such that \eqref{eq:estimate} holds.

Furthermore, we have 
\begin{align}
&\left\|
    C_\ell(\gamma+i\xi)
    \widehat H_\ell(\gamma+i\xi)
    \right\|_{L^2(\mathbb R)}
    \le
     C_{\gamma, R_0}\left(1 + \frac{1} {\sqrt {c_{\ell}}}\right)
    \left(
        \big\| V_\ell^{'[0]}\big\|_{L_\xi^2(I_\ell)}
        +
        \big\| V_\ell^{[1]}\big\|_{L^2(I_\ell)}
    \right), \label{eq:30} \\
 & \left\|
        C_g(\gamma+i\xi)
        \widehat H_g(\gamma+i\xi)
    \right\|_{L_\xi^2(\mathbb R)}
    \le
     C_{\gamma, R_0}\left(1 + \frac{1} {c_{g}}\right)
    \left(
        \big\|V_g^{'[0]}\big\|_{L^2(I_g)}
        +
        \big\| V_g^{[1]}\big\|_{L^2(I_g)}
    \right), \label{eq:31}
\end{align}
The proof of \eqref{eq:30} and \eqref{eq:31} will be given in the Appendix \ref{sec:A}.
Moreover, it can be seen that there exists a fixed threshold $\widehat \gamma_* = \widehat \gamma_*(R_0, c_{\ell,*},\rho_{\ell,*}, c_{g,*},\rho_{g,*},\sigma_0)>0$ such that $\gamma \ge \widehat  \gamma_*$, 
\begin{align}
  & \left\|
        C_\ell(\gamma+i\xi)
    \right\|_{L_\xi^2(\mathbb R)}
    \le
    C_{\gamma,R_0} \frac{1}{1 + |\gamma + i\xi|}, \qquad \left\|
        C_g(\gamma + i\xi)
    \right\|_{L_\xi^2(\mathbb R)}
    \le
    C_{\gamma,R_0} \frac{1}{1+|\gamma + i\xi|}. \label{eq:32}
\end{align}
Let 
\begin{align*}
   \widetilde \gamma^* = \max(\gamma^*, \widehat \gamma^*).
\end{align*}
Combining \eqref{eq:20}, \eqref{eq:30}, \eqref{eq:31} and \eqref{eq:32} yields that
\begin{align} \label{eq:26}
\|\hat \zeta \|_{L_\xi^2(\mathbb R)}
\le C _\gamma \|\mathcal I^{(0)}\|_{\mathcal X^{(0)}_0}, \quad \mathrm{for}\; \gamma > \widetilde \gamma^*.
\end{align}
Furthermore, it follows from \eqref{eq:21} that 
\begin{align} \label{eq:55}
 \hat V_\ell(s) = -C_\ell(s)(-s\hat \zeta(s) + \zeta^{(0)} - \hat H_\ell(s)), \quad  \hat V_g(s) = C_g(s)(-s\hat \zeta(s) + \zeta^{(0)} - \hat H_g(s))
\end{align}
Combining this with \eqref{eq:30}, \eqref{eq:31} and \eqref{eq:32} gives 
\begin{align}\label{eq:47}
\|\hat V_g \|_{L_\xi^2(\mathbb R)} + \|\hat V_\ell \|_{L_\xi^2(\mathbb R)}
\le C_{\gamma,R_0} \|\mathcal I^{(0)}\|_{\mathcal X^{(0)}_0}, \quad \mathrm{for}\; \gamma > \widetilde \gamma^*.
\end{align}

Now we define
\begin{align}
    &\widehat v_g(s,r)
    =
    \frac{\sinh(sr/c_g)}{\sinh(sR_0/c_g)}
    \widehat V_g(s)
    +
    \widehat v_g^{\,f}(s,r),
    \qquad
    0<r<R_0, \label{eq:23}\\
    &\widehat v_\ell(s,r)
    =
    e^{-s(r-R_0)/c_\ell}\widehat V_\ell(s)
    +
    \widehat v_\ell^{\,f}(s,r),
    \qquad
    r>R_0, \label{eq:24}\\
   & \hat\eta(s) = \frac{\hat \zeta(s) + \eta^{(0)}}{s}. \label{eq:25}
\end{align}
Using \eqref{eq:26} and \eqref{eq:25}, we obtain 
\begin{align}
    \|\hat \eta^{(0)}\|_{L_\xi^2(\mathbb R)}
    \le  C_{\gamma,R_0} \|\mathcal I^{(0)}\|_{\mathcal X^{(0)}_0}, \quad \mathrm{for}\; \gamma > \widetilde \gamma^*. \notag
\end{align}

Since for $\alpha \in \left\{\ell, g\right\}$,
\begin{align*}
\left\| \hat v_\alpha
    \right\|_{L^2_\gamma(0,\infty;L^2(I_\alpha))}
    \le C_{\gamma, R_0}
     \left(
        \big\| V_\alpha^{'[0]}\big\|_{L^2(I_\alpha)}
        +
        \big\| V_\alpha^{[1]}\big\|_{L^2(I_\alpha)}
    \right), 
\end{align*}
and 
\begin{align*}
    \left\|
        \frac{\sinh(sr/c_g)}{\sinh(sR_0/c_g)}
    \right\|_{L^2(I_g)}
    \le C_\gamma,
    \qquad
    \left\|
        e^{-s(r-R_0)/c_\ell}
    \right\|_{L^2(I_\ell)}
    \le C_\gamma, 
\end{align*}
applying \eqref{eq:47}, \eqref{eq:23} and \eqref{eq:24} gives
\begin{align*}
    \sum_{\alpha\in\{g,\ell\}}
\| v_\alpha\|_{L^2_\gamma(0,\infty;L^2(I_\alpha))} \le  C_{\gamma,R_0} \|\mathcal I^{(0)}\|_{\mathcal X^{(0)}_0}, \quad \mathrm{for}\; \gamma > \widetilde \gamma^*.
\end{align*}
The above argument proves \eqref{eq:51} and  \eqref{eq:52} hold in the case $m=0$. 

For $ 0 \le j \le m$, we define
\begin{align*}
\widehat{v_{\alpha}^{[j]}}(s)
=
s^j\widehat v_\alpha(s)
-
\sum_{k=0}^{j-1}s^{j-1-k}V_\alpha^{[k]},
\end{align*}
and
\begin{align*}
\widehat{\eta^{[j]}}(s)
=
s^j\widehat\eta^{(0)}(s)
-
\sum_{k=0}^{j-1}s^{j-1-k}E_k.
\end{align*}
Equivalently,
\begin{align*}
\widehat{\eta^{[j]}}(s)
=
\frac{\widehat{\zeta^{[j]}}(s)+E_j}{s}.
\end{align*}
Clearly, 
\[
s^2\widehat{v_{\alpha}^{[j]}}
-
c_\alpha^2\partial_{rr}\widehat{v_{\alpha}^{[j]}}
=
sV_\alpha^{[j]} + V_\alpha^{[j+1]}.
\]
Moreover, the compatibility conditions in the definition of
$\mathcal X^{(0)}_m$ imply that these differentiated data satisfy the same
boundary and interface conditions. Therefore the $m=0$ estimate can be
applied to each differentiated system. 

Now we prove uniqueness in the case $m=0$. Since the higher-order solution classes are subspaces of the $m=0$ class, this suffices. Let two weighted solutions be given and denote their difference by
\begin{align*}
    \bigl(
        v_g^{\mathrm{dif}},
        v_\ell^{\mathrm{dif}},
        \eta^{(0),\mathrm{dif}},
        \zeta^{\mathrm{dif}}
    \bigr).
\end{align*}
Then this difference satisfies the homogeneous system with zero initial data. Taking the Laplace transform on the line $s=\gamma+i\xi$, we obtain, for a.e. $\xi\in\mathbb R$,
\begin{align*}
    s^2\widehat v_\alpha^{\mathrm{dif}}
    -
    c_\alpha^2\partial_{rr}\widehat v_\alpha^{\mathrm{dif}}
    =
    0,
    \qquad
    \alpha\in\{g,\ell\},
\end{align*}
in the sense of distributions. Since
$\widehat v_\alpha^{\mathrm{dif}}(s,\cdot)\in L^2(I_\alpha),$
it follows that
\begin{align*}
    \partial_{rr}\widehat v_\alpha^{\mathrm{dif}}
    =
    \frac{s^2}{c_\alpha^2}
    \widehat v_\alpha^{\mathrm{dif}}
    \in L^2(I_\alpha).
\end{align*}
Hence $
    \widehat v_\alpha^{\mathrm{dif}}(s,\cdot)\in H^2_{\mathrm{loc}}(I_\alpha),
$
and the boundary traces are well-defined.
Therefore the homogeneous transformed solutions have the form
\begin{align*}
    &\widehat v_g^{\mathrm{dif}}(s,r)
    =
    \frac{\sinh(sr/c_g)}{\sinh(sR_0/c_g)}
    \widehat V_g^{\mathrm{dif}}(s),
    \qquad
    0<r<R_0,\\
\mathrm{and}\qquad
    &\widehat v_\ell^{\mathrm{dif}}(s,r)
    =
    e^{-s(r-R_0)/c_\ell}
    \widehat V_\ell^{\mathrm{dif}}(s),
    \qquad\;\;\;
    r>R_0,
\end{align*}
where
\begin{align*}
    \widehat V_\alpha^{\mathrm{dif}}(s)
    :=
    \widehat v_\alpha^{\mathrm{dif}}(s,R_0).
\end{align*}
Here the growing exterior mode is excluded by the condition
$\widehat v_\ell^{\mathrm{dif}}(s,\cdot)\in L^2(I_\ell).$
Substituting these expressions into the transformed interface conditions yields the homogeneous scalar equation
\begin{align*}
    M(s)\widehat\zeta^{\mathrm{dif}}(s)=0.
\end{align*}
By the lower bound for $M$ (Lemma \ref{le:multiplier}), we have
\begin{align*}
    \widehat\zeta^{\mathrm{dif}}(\gamma+i\xi)=0 \quad \textrm{for a.e.}\; \xi\in\mathbb R.
\end{align*}
 Since the initial data of the difference vanish, we also have
\begin{align*}
    \widehat{\eta^{(0),\mathrm{dif}}}(s)
    =
    \frac{\widehat\zeta^{\mathrm{dif}}(s)}{s}
    =
    0.
\end{align*}
The jump condition and the remaining algebraic equations then imply
\begin{align*}
    \widehat V_g^{\mathrm{dif}}(s)
    =
    \widehat V_\ell^{\mathrm{dif}}(s)
    =
    0.
\end{align*}
Consequently,
\begin{align*}
    \widehat v_g^{\mathrm{dif}}(s,\cdot)
    =
    \widehat v_\ell^{\mathrm{dif}}(s,\cdot)
    =
    0.
\end{align*}
By the weighted Plancherel theorem for the Laplace transform, we conclude that
\begin{align*}
    v_g^{\mathrm{dif}}
    =
    v_\ell^{\mathrm{dif}}
    =
    \eta^{(0),\mathrm{dif}}
    =
    \zeta^{\mathrm{dif}}
    =
    0.
\end{align*}
Hence the weighted solution is unique.

Finally, we prove \eqref{eq:53}.
Building upon the construction,  
\begin{align}
s^2\widehat v_\ell - c_\ell^2\partial_{rr}\widehat v_\ell = s  V_\ell^{[0]} +  V_\ell^{[1]}, \quad  r > R_0, \quad \lim_{r\rightarrow \infty} \widehat v_\ell = 0, \label{eq:40}\\
    s^2\widehat v_g - c_g^2\partial_{rr}\widehat v_g = s  V_g^{[0]} +  V_g^{[1]}, \quad 0< r< R_0, \label{eq:41}\\
\widehat v_\ell(R_0,s) - \widehat v_g(R_0,s) = \frac{2\sigma}{R_0} \frac{\hat \zeta(s) + \eta^{(0)}}s, \label{eq:38}\\
B_\ell \hat v (R_0,s) =  B_g\hat v(R_0,s). \notag
\end{align}
Let $\hat w_\alpha(r,s):= s\hat v_\alpha(r,s) - p_{\ell,0}$. In conjunction with the compatible condition \eqref{eq:37} and \eqref{eq:38}, we have 
\begin{align} \label{eq:39}
\hat w_\ell(R_0,s) - \hat w_g(R_0,s) = \frac{2\sigma}{R_0} \zeta(s). 
\end{align}
Multiplying $\overline {\hat w_{\alpha}}$ on both sides of equations \eqref{eq:40} and \eqref{eq:41}, and using \eqref{eq:21}, \eqref{eq:22} and \eqref{eq:39}, we obtain 
\begin{align*}
&\sum_{\alpha \in \{\ell,g\} }s\frac{1}{k_\alpha}\int_{I_\alpha} |w_\alpha(r)|^2dr + \frac{1}{\rho_\alpha}\int_{I_\alpha} \overline s|\partial_r \hat v_\alpha(r)|^2 - \partial_r \hat v_\alpha(r) \overline { V_\alpha^{[0]}}dr = \sum_{\alpha \in \{\ell,g\} }\int_{I_\alpha}  V_\alpha^{[1]}(r) \overline{\hat w_\alpha}(r)dr\\
& + \frac{1}{R_0 \rho_g}\hat V_g(s)\hat w _g(R_0,s) - \frac{1}{R_0 \rho_g}\hat V_\ell(s)\hat w _\ell(R_0,s)  -\frac{2\sigma}{R_0} \overline{\hat \zeta(s)} (-s \hat \zeta(s) +\zeta^{(0)}).
\end{align*}
Dividing by $\overline s$ and using Yong inequality, we have 
\begin{align*}
    \frac{1}{\rho_\alpha} \|\partial_r\hat v_\alpha(\cdot, s)\|_{L^2(I_\alpha)}  &  \le C_{\gamma, R_0}\bigg[\sum_{\alpha \in \{\ell,g\}}\left(\frac{1}{k_\alpha}\|\hat w_\alpha(\cdot,s)\|_{L^2(I_\alpha)} + \frac{1}{\rho_\alpha} \|\hat V_\alpha\|_{L_\xi^2(\mathbb R)}\right)\\
    &\qquad\qquad+ \|\hat \zeta \|_{L_\xi^2(\mathbb R)} + \left| I \right|_{\mathcal X^{(0)}_0}\bigg].
\end{align*}
Together with the validity of \eqref{eq:52} in the case $m=1$, the above estimate
implies \eqref{eq:53} for $m=1$. The case $m\ge 2$ follows by the same argument.
\end{proof}

\begin{remark}[Uniform estimate for the reduction]\label{rem:uniform-estimate-reduction}
Let $m \ge 1$, and let $R_0, \rho_{\ell,*},c_{\ell,*}, \sigma_*  >0$.
Assume that $\rho_\ell\ge\rho_{\ell,*}>0,\;
c_\ell\ge c_{\ell,*}>0$, $0<\sigma\le \sigma_*$ and that $c_{g}$ is sufficiently large.
In the proof of Theorem \ref{thm:wellposed-B}, it was shown that the estimate
\begin{align} \label{eq:54}
 &\left\|\frac{d}{dt} \eta^{(0)}\right\|_{H^m_\gamma}
+  \|\eta^{(0)}\|_{H^{m}_\gamma} \le
C_{\gamma,R_0}\|\mathcal I^{(0)}\|_{\mathcal X^{(0)}_m},\qquad \mathrm{for}\; \gamma > \widetilde \gamma_*,
\end{align}
We now
introduce a subset $\widetilde{\mathcal X}^{(0)}_m \subset \mathcal X^{(0)}_m$:
\begin{align*}
  \widetilde{\mathcal X}^{(0)}_m  := \left\{\mathcal I^{(0)} \in \mathcal X^{(0)}_m: E_{j_1} = 0,\; V^{[j_2]}_\ell = V^{[j_2]}_g = 0,\;  1\le j_1 \le m + 1,\; 2\le j_2 \le m + 1\right\}. 
\end{align*}
For initial data in $\widetilde{\mathcal X}^{(0)}_m$, since the constants in the estimates
\eqref{eq:30} and \eqref{eq:31}, together with the multiplier estimate, are
uniform with respect to $c_\alpha$ and $c_\alpha$ $(\alpha \in \{\ell,g\})$. Therefore, by repeating the
argument used in the proof of \eqref{eq:54}, we obtain the same
estimate with a constant independent of $c_\alpha$, $\rho_\alpha$ and $\sigma$, as
$c_g,c_\ell\to\infty$. Based on this fact, using the estimates 
\eqref{eq:30}, together with uniform boundness of $C_\ell(s)$ with respect to $c_\ell$, and applying \eqref{eq:55}, we obtain that 
\begin{align*}
\|v_\ell(R_0,\cdot)\|_{H^{m}_\gamma} \le
C_{\gamma,R_0}\|\mathcal I^{(0)}\|_{\mathcal X^{(0)}_{m+2}},\qquad \mathrm{for}\; \gamma > \gamma_*.   
\end{align*}
Here, $\gamma_*:=\gamma_*(R_0, \rho_{\ell,*},c_{\ell,*},\sigma_*)>0$. Combining this with the condition 
\begin{align*}
  \partial^j_t v_\ell(R_0,t) - \partial^j_t v_g(R_0,t) = \frac{2\sigma}{R_0} \partial^j_t \eta^{(0)}(t), \quad 0 \le j \le m,
\end{align*}
we obtain 
\begin{align*}
\|v_g(R_0,\cdot)\|_{H^{m}_\gamma} \le
C_{\gamma,R_0}\|\mathcal I^{(0)}\|_{\mathcal X^{(0)}_{m+2}},\qquad \mathrm{for}\; \gamma > \gamma_*.   
\end{align*}
This uniform estimate will be used in the subsequent reduction argument. 
\end{remark}

\subsection{Semigroup well-posedness on the orthogonal complement}
\label{sec:wellposdeness_2}

This section is devoted to proving the well-posedness of the component orthogonal to the
spherically symmetric mode. Throughout this subsection, we use the modal
notation introduced in Subsection~\ref{subsec:spherical-harmonics}, and all bulk
and surface variables are restricted to the \(\Pi_{\ge1}\)-sector. 
For $s\ge 0$,
we define
\begin{align*}
H_{\rm bulk}^{s,\ge1} := \left\{u\in H^s(\Omega_\ell^{(0)})\times H^s(\Omega_g^{(0)}):  \Pi_0 u = 0\right\},
\end{align*}
and the interface spaces
\begin{align*}
H^{(s)}_{\mathrm{sur}}
  := \left\{\eta \in H^s(S^{(0)}): \Pi_0 \eta = 0\right\}.
\end{align*}
We further define the Hilbert space
\begin{align}
\mathcal O_{\ge 1}:= \Big(H_{\mathrm{bulk}}^{(1,\ge 1)}\times H^{0,\ge 1}_{\mathrm{bulk}} \Big)
     \times \Big(H^{(1,\ge 1)}_{\mathrm{sur}}\times H^{(1,\ge 1)}_{\mathrm{sur}} \Big),
 \notag 
\end{align}
and equip it with the inner product
\begin{align}
   &\left\langle U^{(1)}, U^{(2)}\right\rangle_{\mathcal O_{\ge 1}}
 :=\; 
 \sum_{\alpha\in\{\ell,g\}}\int_{\Omega_\alpha^{(0)}}
      \left(
        \frac{1}{K_\alpha}q^{(1)}_\alpha \overline{q_\alpha^{(2)}} 
       + \frac{1}{\rho_{\alpha}}\left[\nabla p^{(1)}_\alpha \cdot \overline{\nabla p^{(2)}_\alpha} + p^{(1)}_\alpha \overline{p^{(2)}_\alpha} \right]
      \right)(x)\,dx \notag
   \\
  & +
     \left(\frac{-2\sigma}{R^{2}_{0}}\left\langle\zeta^{(1)}, \zeta^{(2)}\right\rangle_{L^2\left(S^{(0)}\right)} + \sigma \left\langle\nabla_{S^{(0)}}\zeta^{(1)}, \nabla_{S^{(0)}}\zeta^{(2)}\right\rangle_{L^2\left(S^{(0)}\right)} + \left\langle \eta^{(1)}, \eta^{(2)}\right\rangle_{H^1\left(S^{(0)}\right)}\right), \notag
\end{align}
where $U^{(l)}:=\big(p^{(l)},q^{(l)},\eta^{(l)},\zeta^{(l)})$ for $l = 1,2$.
Define  
\begin{align*}
\mathcal H_{\ge 1}^0:=\left\{(p,q,\eta,\zeta) \in \mathcal O_{\ge 1}:\Pi_{1}[p] = 0,\;
\left(\Delta_{S^{(0)}}+\frac{2}{R_{0}^2}\right)^{-1}\Pi_{\ge 2}[p] = \sigma\Pi_{\ge 2} \eta\; \mathrm{on}\; S^{(0)} \right\}.
\end{align*}
Now, we define the state space $\mathcal H_{\ge 1}$ as the completion of $\mathcal H^0_{\ge 1}$ with respect to the norm induced by the inner product $\langle\cdot,\cdot\rangle_{\mathcal O_{\ge 1}}$.

For the bulk operators, let 
\begin{align}
  L_\alpha p_\alpha
  &:= K_\alpha\,\nabla\cdot\Big(
        \frac{1}{\rho_{\alpha}}\nabla p_\alpha\Big),
    \qquad \mathrm{in}\;\Omega_\alpha^{(0)},\quad \alpha \in \{\ell,g\}, \notag 
\end{align}
and let $L$ denote the corresponding piecewise operator. 

We now define the operator $\mathcal A^{\perp}:D(\mathcal A^{\perp})\subset \mathcal H_{\ge 1}\to\mathcal H_{\ge 1}$ by 
\begin{align*}
\mathscr A_{\ge1}
\begin{pmatrix}
p\\ q\\ \eta\\ \zeta
\end{pmatrix}
:=
\begin{pmatrix}
q\\
Lp\\
\zeta\\
-\partial_n p
\end{pmatrix}.
\end{align*}
The domain \(\mathcal D(\mathscr A_{\ge1})\) consists of all
\(U=(p,q,\eta,\zeta)\in\mathcal H_{\ge1}\) such that
\begin{enumerate}[(1)]
\item
\begin{align*}
Lp\in H_{\rm bulk}^{0,\ge1},
\qquad
q\in H_{\rm bulk}^{1,\ge1},
\qquad
\eta,\zeta\in H_{\rm sur}^{5/2,\ge1},
\end{align*}
\item the \emph{jump condition} holds:
\begin{align}
&\Pi_{1}[q] = 0, \label{eq:jump-split1} \\
&\left(\Delta_{S^{(0)}}+\frac{2}{R_{0}^2}\right)^{-1}\Pi_{\ge 2}[q] = \sigma \Pi_{\ge 2} \zeta, \label{eq:jump-split2}
\end{align}
\item and the \emph{flux condition} holds:
\begin{align*}
  \frac{1}{\rho_{\ell}}\partial_{n} p_\ell
  = \frac{1}{\rho_{g}}\partial_{n} p_{g}
  \quad\text{on }S^{(0)}, \qquad  \partial_n p \in H_{\rm sur}^{1,\ge1}.
\end{align*}
\end{enumerate}
With this notation, the $\Pi_{\ge1}$-component of the pressure--interface
system is written as
\begin{align}\label{eq:A-pt}
\partial_t U=\mathscr A_{\ge1}U .
\end{align}

Now we state the well-posedness result in the orthogonal complement.

\begin{theorem}[Well-posedness of the orthogonal complement]
\label{thm:wellposed-A}
We have the following arguments.
\begin{enumerate}[(a)] 
\item \label{c1} The operator $\mathcal A^{\perp}$ with domain $D(\mathcal A^{\perp})$ generates
a $C^0$-semigroup $\{\mathcal S(t)\}_{t\ge 0}$ on $\mathcal H_{\ge 1}$. For every $U_0\in\mathcal H_{\ge 1}$,
there exists a unique mild solution 
\begin{align*}
U(t) = \mathcal S(t)U_0\in C([0,\infty);\mathcal H_{\ge 1}),
\end{align*}
and for $U_0\in D(\mathcal A^{\perp})$ the solution is classical, namely, 
\begin{align*}
U\in C^1([0,\infty);\mathcal H_{\ge 1})\cap C([0,\infty);D(\mathcal A^{\perp})).
\end{align*}
\item \label{c2}
Let $U_0\in D(\mathcal A^{\perp})$.
We have the energy {conservation}
\begin{align}
&\left\|\mathcal S(t)U_0\right\|_{\mathcal E} = \|U_0\|_{\mathcal E}\qquad t\ge0. \label{eq:gronwall-step2}
\end{align}
Here,
\begin{align*}
\|U\|^2_{\mathcal E}&:= \|U\|^2_{\mathcal O_{\ge 1}} - \sum_{\alpha\in\{\ell,g\}}\frac{1}{\rho_{\alpha}} \|p_\alpha\|^2_{L^2\left(\Omega_\alpha^{(0)}\right)} - \left\|\eta\right\|^2_{H^1\left(S^{(0)}\right)}.
\end{align*}
\end{enumerate}
\end{theorem}

\begin{proof}
\eqref{c1}
We verify the hypotheses of the Lumer--Phillips theorem in two steps;  see, for instance, \cite{Pazy1983}.

\smallskip\noindent
\emph{Step 1: Dissipativity.}
For each $U \in D(\mathcal A^{\perp})$, a straightforward calculation gives 
\begin{align}
\langle \mathcal A^{\perp} U, &U\rangle_{\mathcal H_{\ge 1}}
:=\;\sum_{\alpha\in\{\ell,g\}}\int_{\Omega_\alpha^{(0)}}
   \left(
    \nabla\cdot\bigg(
        \frac{1}{\rho_{\alpha}}\nabla p_\alpha\bigg)\overline q_\alpha
       + \frac{1}{\rho_{\alpha}}\nabla q_\alpha\overline{\nabla p_\alpha} + \frac{1}{\rho_{\alpha}}q_\alpha \overline p_\alpha
      \right)(x)\,dx \notag
   \\
 + &\Big(\int_{S^{(0)}}\left(\zeta \overline{\eta} + \nabla_{S^{(0)}}{\zeta} \nabla_{S^{(0)}}\overline{\eta}\right)(x) dS(x) \notag\\
 & + \frac{\sigma}{R^{2}_{0}} \int_{S^{(0)}}\left(\partial_{n}p\overline{\zeta}\right)(x) dS(x)\Big) - \sigma \int_{S^{(0)}}\left(\nabla_{S^{(0)}}\partial_{n}p\overline{\nabla_{S^{(0)}} \zeta}\right)(x) dS(x)\Big). \label{eq:1}
\end{align}
With the aid of the jump conditions \eqref{eq:jump-split1} and \eqref{eq:jump-split2}, by integrating by parts, we have that 
\begin{align}
&\mathrm{Re} \Bigg[ \sum_{\alpha\in\{\ell,g\}}\int_{\Omega_\alpha^{(0)}}
      \left(
        \nabla\cdot\bigg(
        \frac{1}{\rho_{\alpha}}\nabla p_\alpha\bigg)\overline q_\alpha
       + \frac{1}{\rho_{\alpha}}\nabla q_\alpha\overline{\nabla p_\alpha}
      \right)(x)\,dx 
    +\frac{\sigma}{R^{2}_{0}}\int_{S^{(0)}}\left(\partial_{n}p\overline{\zeta}\right)(x) dS(x)\Bigg] \notag \\
&= \mathrm{Re}\left[ \int_{S^{(0)}}\left(\partial_{n}p\overline{(q_{g}-q_\ell)} + \frac{\sigma}{R^{2}_{0}}\partial_{n}p\overline{\zeta}\right)(x)dS(x)\right] = \mathrm{Re}\Bigg[ - \sigma \int_{S^{(0)}_j}\left(\partial_{n}p\overline{\Delta_{S^{(0)}} \zeta}\right)(x) dS(x)\Bigg], \label{eq:2}
\end{align}
and that 
\begin{align*}
-\int_{S^{(0)}}\left(\nabla_{S^{(0)}}\partial_{n}p\overline{\nabla_{S^{(0)}}\zeta}\right)(x) dS(x)  =
\int_{S^{(0)}}\left(\partial_{n}p\overline{\Delta_{S^{(0)}} \zeta}\right)(x) dS(x).
\end{align*}
Moreover, by Cauchy-Schwartz inequality, we have 
{
\begin{align*}
&\left|\int_{S^{(0)}}\left(\zeta \overline \eta + \nabla_{S^{(0)}}{\zeta} \overline{\nabla_{S^{(0)}}\eta}\right)(x) dS(x)\right|
\le \|\zeta\|_{H^1(S^{(0)})}\|\eta\|_{H^1(S^{(0)})}.
\end{align*}
}
In conjunction with \eqref{eq:1} and \eqref{eq:2}, we arrive at 
\begin{align}
\mathrm{Re}(\langle \mathcal A^{\perp} U, U\rangle) = 0, \qquad \mathrm{for}\; U \in D(\mathcal A^{\perp}).  \notag
\end{align}

\smallskip\noindent
\emph{Step 2: Maximality (range condition).}
In this step, we aim to find $\lambda > 0$, such that 
\begin{align*}
\mathrm{Ran}(\lambda \mathbb I - \mathcal A^{\perp}) = \mathcal H_{\ge 1},
\end{align*}
that is, given $F=(f_p,f_q,f_\eta, f_\zeta)\in\mathcal H_{\ge 1}$, there exists $U=(p,q,\eta,\zeta)\in D(\mathcal A^{\perp})$ satisfying 
\begin{align}
&q=\lambda p-f_p, \qquad (\lambda^2 \mathbb I-L)p = f_q+\lambda f_p 
\quad\text{in } \mathbb R^3 \setminus S^{(0)},
\label{eq:res-q-bulk} \\
&\mathrm{and}\;\zeta = \lambda \eta - f_\eta, \qquad \lambda^2 \eta + \partial_{n} p = f_\zeta + \lambda f_\eta\quad \mathrm{in}\; S^{(0)}. \label{eq:res-q-interface}
\end{align}

\smallskip\noindent
\underline{Variational formulation.} 
Let 
\begin{align*}
H_{\mathrm{bulk}}^{\mathrm{tr}}:=\left\{p\in H_{\mathrm{bulk}}^{(1,\ge 1)}: \Pi_{1}[p] = 0\; \mathrm{on}\; S^{(0)}\right\},
\end{align*}
and define 
$a_{\lambda}:H_{\mathrm{bulk}}^{\mathrm{tr}}\times H_{\mathrm{bulk}}^{\mathrm{tr}}\to\mathbb{C}$  by
\begin{align*}
a_{\lambda}(p,\phi):={} & \sum_{\alpha\in\{\ell,g\}}\int_{\Omega_\alpha^{(0)}}\Big(\frac{\lambda^2}{K_\alpha}p_\alpha\overline{\phi_\alpha}
+\frac{1}{\rho_{\alpha}}\nabla p_\alpha\cdot \overline{\nabla\phi_\alpha}\Big)(x)\,dx \\
& - \frac{\lambda^2}{\sigma}
\left\langle
\left(\Delta_{S^{(0)}}+\frac{2}{R_{0}^2}\right)^{-1}\Pi_{\ge 2}[p]
\bigr),
[\phi]
\right\rangle_{L^2(S^{(0)})}.
\end{align*}
Since
\begin{align} \label{eq:harmonic}
(\Delta_{\mathbb S^2} + 2) Y_l^m = (l+2)(1-l)Y_l^m, \quad \mathrm{for}\; l\ge 2,
\end{align}
it can be seen that $a_\lambda$ is coercive in $H_{\mathrm{bulk}}^{\mathrm{tr}}$. 
Since $f_p \in H^{\mathrm{rm}}_{\mathrm{bulk}}$, $f_q \in H^{\mathrm{tr}}_{\mathrm{bulk}}$, $f_\eta \in H_{\mathrm{sur}}^{(1,\ge 1)}$ and $f_\zeta \in H^{(1,\ge 1)}_{\mathrm{sur}}$, we can find a unique $F \in \left(H^{\mathrm{rm}}_{\mathrm{bulk}}\right)'$ satisfying
\begin{align*}
F(\phi) = \left\langle \widetilde f_\zeta , [\phi]\right\rangle_{L^2(S^{(0)})} +  \left\langle f_q + \lambda f_p,\phi\right\rangle_{H_{\mathrm{bulk}}}, \quad {\phi} \in H^{\mathrm{tr}}_{\mathrm{bulk}}. 
\end{align*}
Here, 
\begin{align} \label{eq:15}
\widetilde f_\zeta: = f_\zeta + \frac{1}{\sigma}\left(\Delta_{S^{(0)}} + \frac{2}{R_{0}^2}\right)^{-1}\Pi_{\ge 2}[f_p] \quad \mathrm{on}\; S^{(0)}.
\end{align}
By Lax-Milgram Theorm, there exists $\lambda > 0$ such that there exists a unique $p\in H_{\mathrm{bulk}}^{\mathrm{tr}}$ satisfying
\begin{align} \label{eq:3}
a_\lambda(p,\phi) = F(\phi), \qquad \phi \in H^{\mathrm{tr}}_{\mathrm{bulk}}.
\end{align}
Then, choosing any $\phi_\ell \in H^1_0(\Omega_\ell)$ and $\phi_g = 0$, we have 
\begin{align*}
&\int_{\Omega_\ell^{(0)}}\Big(\frac{\lambda^2}{K_\ell}p_\ell \overline{\phi_\ell}+\frac{1}{\rho_{\ell}}\nabla
    p_\ell\cdot\overline{\nabla\phi_\ell}\Big)(x)\,dx = \int_{\Omega_\ell^{(0)}} \left((f_q + \lambda f_p)\overline{\phi_\ell}\right)(x) dx,\quad  \mathrm{for}\; \phi_\ell \in H^1_0(\Omega_\ell).
\end{align*}
Therefore, we obtain that $\partial_{n} p_\ell \in H^{-1/2}(S^{(0)})$ and that 
\begin{align}
&\int_{\Omega_\ell^{(0)}}\Big(\frac{\lambda^2}{K_\ell}p_\ell \overline{\phi_\ell} +\frac{1}{\rho_{\ell}}\nabla
p_\ell\cdot \overline{\nabla\phi_\ell}\Big)(x)\,dx  = \int_{\Omega_\ell^{(0)}} \left((f_q + \lambda f_p)\overline{\phi_\ell}\right)(x) dx \notag \\
& - \int_{S^{(0)}} \frac{1}{\rho_\ell}\left(\partial_{n} p_\ell \overline{\phi_\ell}\right)(x) dS(x),\quad  \mathrm{for}\; \phi_\ell \in H^1(\Omega_\ell).\label{eq:12}
\end{align}
Similarly, we have that $\partial_{n} p_{g} \in H^{-1/2}(S^{(0)})$ and that 
\begin{align}
&\int_{\Omega_g^{(0)}}\Big(\frac{\lambda^2}{K_{g}}p_{g}\overline{\phi_{g}} +\frac{1}{\rho_{g}}\nabla p_{g}\cdot\overline{\nabla\phi_{g}}\Big)(x) dx = \int_{B^{(0)}}\left((f_q + \lambda f_p)\overline{\phi_{g}}\right)(x) dx \notag \\
& + \int_{S^{(0)}} \frac{1}{\rho_{g}}\left(\partial_{n} p_{g}\overline{ \phi_{g}}\right)(x) dS(x),\quad  \mathrm{for}\; \phi_{g} \in H^1(S^{(0)}). \label{eq:13}
\end{align}
Combining \eqref{eq:3}, \eqref{eq:12} and \eqref{eq:13} gives
\begin{align}
&-\frac{\lambda^2}\sigma \left\langle\left(\Delta_{S^{(0)}}+\frac{2}{R_{0}^2}\right)^{-1} \Pi_{\ge 2}[p], [\phi]\right\rangle_{L^2(S^{(0)})} \notag\\
&= 
\int_{S^{(0)}} \frac{1}{\rho_\ell} \left(\partial_{n} p_\ell \overline{\phi_\ell} \right)(x) - \frac{1}{\rho_{g}}\left(\partial_{n} p_{g} \overline{\phi_{g}}\right)(x) dS(x) - \left\langle\widetilde f_\zeta , [\phi]\right\rangle_{L^2(S^{(0)})}. \notag
\end{align}
From this, we have 
\begin{align}
&\frac{1}{\rho_\ell}\partial_{n} p_\ell = \frac{1}{\rho_{g}}\partial_{n} p_{g}, \notag \\
& \frac{\lambda^2}\sigma \left(\Delta_{S^{(0)}}+\frac{2}{R_{0}^2}\right)^{-1}\Pi_{\ge 2}[p] = -\Pi_{\ge 2}(\partial_{n} p - \widetilde f_\zeta)\quad \mathrm{on}\; S^{(0)}. \label{eq:4}
\end{align}

Now we recover remaining components from $p$. Set 
\begin{align}
    \zeta := \frac{\lambda}{\sigma}\left(\Delta_{S^{(0)}}+\frac{2}{R_{0}^2}\right)^{-1}\Pi_{\ge 2}\left([p]- \frac{1}\lambda [f_p]\right) 
    \label{eq:6}
\end{align}
and
\begin{align}
\eta :=  
\frac{1}{\lambda} \zeta + \frac{1}{\lambda}f_\eta.
\label{eq:7}
\end{align}
Combing \eqref{eq:6}, \eqref{eq:7} and the fact that $(f_p,f_q,f_\eta, f_\zeta) \in \mathcal H_{\ge 1}$ gives
\begin{align}\label{eq:16}
\frac{1}{\sigma}\left(\Delta_{S^{(0)}} + \frac{2}{R_{0}^2}\right)^{-1}\Pi_{\ge 2}[p] = \Pi_{\ge 2} \eta.
\end{align}
Using \eqref{eq:15}, \eqref{eq:4} and \eqref{eq:6}, we obtain 
\begin{align*}
 \lambda \zeta =  -(\partial_{n} p - f_\zeta) \quad \mathrm{on}\; S^{(0)}.
\end{align*}
In conjunction with \eqref{eq:6} and \eqref{eq:7}, we arrive at 
\begin{align*}
&\zeta = \lambda \eta - f_\eta,\\
&\lambda^2 \eta + \partial_{n} p= f_\zeta + \lambda f_\eta.
\end{align*}
Set
\begin{align} \label{eq:8}
q = \lambda p - f_p.
\end{align}
This, together with \eqref{eq:4} and \eqref{eq:6} gives
\begin{align*}
\frac{1}{\sigma}\left(\Delta_{S^{(0)}}+\frac{2}{R_{0}^2}\right)^{-1}\Pi_{\ge 2}[q] = \Pi_{\ge 2} \zeta.
\end{align*}
Moreover, building upon \eqref{eq:6}, \eqref{eq:7}, \eqref{eq:16}, \eqref{eq:8},
and the regularity properties of elliptic operators, we obtain that 
$U:=(p,q,\eta,\zeta) \in D(\mathcal A^{\perp})$ and that $U$ satisfies \eqref{eq:res-q-bulk} and \eqref{eq:res-q-interface}.
Thus $(p,q,\eta,\zeta)$ is the desired solution, and the range condition is therefore satisfied.

\eqref{c2}
Let $U(t):= \mathcal S(t)U_0$. It follows from statement \eqref{a1} of this lemma that there exists a unique $U(t) \in C^1([0,\infty);\mathcal H_{\ge 1})\cap C([0,\infty);D(\mathcal A^{\perp}))$ 
satisfying
\begin{align*}
\partial_t U = A^{\perp}U\quad \mathrm{and} \quad U(0)=U \in D(\mathcal A^{\perp}).
\end{align*} 

We first calculate $\frac12\frac{d}{dt}\|U(t)\|_{\mathcal E}^2$.

\smallskip\noindent
\underline{Bulk part.}
Since we have 
\begin{align*}
\partial_t p_{\alpha} = q_\alpha \quad \mathrm{and}\quad \partial_t q_{\alpha}= L_\alpha p_\alpha.
\end{align*}
Therefore, we arrive at 
\begin{align}
&\sum_{\alpha\in\{\ell,g\}}\frac{d}{dt}\int_{\Omega_\alpha^{(0)}}\Big(\frac{1}{K_\alpha}|q_\alpha|^2+\frac{1}{\rho_{\alpha}}|\nabla p_\alpha|^2\Big)(x,t)\,dx \notag \\
&= \sum_{\alpha\in\{\ell,g\}} 2 \mathrm{Re}\left[\int_{\Omega_\alpha^{(0)}}\Big(\frac{1}{K_\alpha}\overline{q_\alpha}{q_{\alpha,t}}
+\frac{1}{\rho_{\alpha}}\nabla p_\alpha\cdot \overline{\nabla p_{\alpha,t}}\Big)(x,t)\,dx\right] \notag \\
& = \sum_{\alpha\in\{\ell,g\}}2\mathrm{Re}\left[\int_{\Omega_\ell^{(0)}}\Big(\overline{q_\ell}\,\nabla\!\cdot\!\left(\frac{1}{\rho_{\ell}}\nabla p_\ell\right)
+ \frac{1}{\rho_{\ell}}\nabla p_\ell\cdot \overline{\nabla q_\ell}\Big)(x,t)\,dx\right] \notag \\
& = -2\mathrm{Re}\left[\int_{S^{(0)}} \frac{1}{\rho_\ell}\left(\partial_n p_{\ell}\,\overline{[q]}\right)(x,t)\,dS(x)\right]
\label{eq:bulk-interface-flux-step1}.
\end{align}
Furthermore, using 
\begin{align}
&\frac{1}{\rho_{\ell}}\partial_{n}p_\ell(t) = - \partial_t\zeta(t)\quad \mathrm{and} \notag \\
&\frac 1{\sigma}\left(\Delta_{S^{(0)}}+\frac{2}{R_{0}^2}\right)^{-1}\left(\Pi_{\ge 2}\left[q(t)\right]\right) = \Pi_{\ge 2} \zeta(t) \quad \mathrm{on}\;S^{(0)}, \notag
\end{align}
we obtain that 
\begin{align}
-\int_{S^{(0)}}\left(\frac{1}{\rho_{\ell}}\partial_{n}p_\ell\,\overline{[q]}\right)(x,t)\,dS(x) & = -{\sigma}\int_{S^{(0)}} \left(\overline{\nabla_{S^{(0)}}\zeta} \cdot \nabla_{S^{(0)}} \partial_t\zeta\right)(x,t) dS(x) \notag  \\
& + \frac{2\sigma}{R^2_{0}}\int_{S^{(0)}} \frac{1}{\rho_{\ell}}\left(\partial_{n}p_{\ell}\overline{\Pi_{\ge 2}\zeta}\right)(x,t) \,dS(x). \label{eq:9}
\end{align}

\smallskip\noindent
\underline{Surface part (high modes).}
We note that 
\begin{align*}
\partial_t \eta = \zeta\quad \mathrm{on}\quad S^{(0)},
\end{align*}
and that \eqref{eq:8} holds.
Therefore, a straightforward calcultaion gives
\begin{align}
&\frac{d}{dt}
\left\|\zeta\right\|_{L^2\left(S^{(0)}\right)}^2 = 2\mathrm{Re}\left[\int_{S^{(0)}} \frac{1}{\rho_\ell}\left(\partial_{n}p_{\ell} \overline{\Pi_{\ge 2} \zeta}\right)(x,t) \,dS(x)\right],
\label{eq:17}\\
&\frac{d}{dt} \left\|\nabla_{S^{(0)}}\zeta\right\|_{L^2\left(S^{(0)}\right)}^2 = 2 \mathrm{Re}\left[\int_{S^{(0)}} \left(\overline{\nabla_{S^{(0)}}\zeta} \cdot \nabla_{S^{(0)}} \partial_t\zeta\right)(x,t) dS(x)\right].\label{eq:10}
\end{align}
Combining \eqref{eq:bulk-interface-flux-step1},\eqref{eq:9}, \eqref{eq:17} and \eqref{eq:10} gives
\eqref{eq:gronwall-step2}.
\end{proof}

\section{Reduced models and asymptotic regimes} \label{sec:reduced-models}

In this section, building on the well-posedness theory developed in Section \ref{sec:wellposedness}, we discuss several reduced models and asymptotic regimes of
the linearized fluid--bubble system. We first relate the present model, in a
suitable asymptotic regime, to the classical linearized Rayleigh--Plesset equation and to
its Lamb-type generalizations; see Theorem \ref{th:reduct_1}. We then consider a small-\(\sigma\) regime leading
to a frozen-interface transmission model; see Theorem \ref{prop:sigma-to-zero}. These two reductions
are presented in the following two subsections. 

Throughout this section, we use the modal projections introduced in
Subsection~\ref{subsec:spherical-harmonics}. The initial datum is denoted by
\begin{align*}
\mathcal I
:=
\left(
p_{\ell,0},p_{g,0},q_{\ell,0},q_{g,0},\eta_0,\zeta_0
\right).
\end{align*}
We decompose it into its spherically symmetric component and its orthogonal
component:
\begin{align*}
\mathcal I = \mathcal I^{(0)}+\mathcal I^{\perp}.
\end{align*}

\subsection{Relation to the classical one-bubble radial models}
We consider the large-wave-speed regime. In this regime, the present
model reduces to the linearized classical one-bubble models, including the
Rayleigh--Plesset equation and its Lamb-type generalizations.
Before proving this, we establish two technical lemmas, whose proofs are deferred to the Appendix \ref{sec:B}.

\begin{lemma} \label{le:2}
Let $p$ be a solution of \eqref{eq:bulk-wave-liquid-final}. Assume that $p\in C^1([0,\infty); H^1(\Omega^{(0)}_\ell \cup \Omega^{(0)}_g))\cap C([0,\infty);H^1_\Delta(\Omega^{(0)}_\ell \cup \Omega^{(0)}_g))$
Here, 
\begin{align*}
H^1_\Delta(\Omega^{(0)}_\ell \cup \Omega^{(0)}_g):= \{H^1(\Omega^{(0)}_\ell \cup \Omega^{(0)}_g)): \Delta p \in L^2(\Omega^{(0)}_\ell \cup \Omega^{(0)}_g))\}.
\end{align*}
We have that 
\begin{align} \label{eq:1-LP}
p(x,t) & = \int_{\mathbb R}\int_{S^{(0)}}\partial_\nu G_{\ell}(x-y,t-s) p_\ell(y,s) - G_{\ell}(x-y,t-s) \partial_\nu p_\ell(y,s) dS(y) ds \notag\\
+ & \frac{1}{c_\ell^2}\left[I^{\rm in}_{\ell}q_{\ell,0} - \partial_t I^{\rm in}_{\ell}p_{\ell,0}\right], \quad x \in \Omega^{(0)}_\ell, \; t \in \mathbb R_+,
\end{align}
and that
\begin{align} \label{eq:2-LP}
p(x,t) & = \int_{\mathbb R}\int_{S^{(0)}}G_{g}(x-y,t-s)\partial_\nu p_g(y,s) - \partial_\nu G_{g}(x-y,t-s) p_g(y,s) dS(y) ds \notag \\
& + \frac{1}{c_g^2}\left[I^{\rm in}_{g}q_{g,0} - \partial_t I^{\rm in}_{g}p_{g,0}\right], \quad x \in \Omega^{(0)}_g,\; t\in \mathbb R_+.
\end{align}
Here, $G_{\alpha}(x,\tau)$ is defined by 
\begin{align*}
G_\alpha(y,t) := \frac{\delta_0(t-c_\alpha^{-1}|x|)}{4\pi|x|}\quad \mathrm{in}\; \mathbb R^3\times \mathbb R_+
\end{align*}
with $\delta_0$ being the Dirac delta distribution, and
\begin{align*}
I^{\rm in}_\alpha h_\alpha(x,t)
:=
\frac{1}{4\pi t}
\int_{\Omega_{\alpha}\cap \{|x-y|=c_\alpha t\}} h_\alpha(y)\,dS(y), \quad\; h_{\alpha} \in H^{(1)}(\Omega_\alpha),\quad \alpha\in \{\ell,g\}.
\end{align*}
\end{lemma}

\begin{lemma}\label{le:3}
Let $T> 0$ be fixed. Let \(h_\alpha\in H_{\Omega_\alpha}^{(1)}\) for $\alpha \in \{\ell,g\}$, and assume in addition that
$h_\ell$ is compactly supported. 
Then, for every
$\psi\in L^{\infty}(0,T)$, there exists a positive constant $C_T$, independent of $c_\alpha$, such that
\begin{align*}
\sup_{t\in[0,T]}\left| 
\int_0^t
\left(I^{\rm in}_\alpha h_\alpha\right)^{(m,l)}(t)\psi(t)\,dt\right| \le C_{T,h_\alpha} \|\psi\|_{L^{\infty}[0,T]}.
\end{align*}
for all sufficiently large $c_\alpha$.
Here, for $g\in L^2(S^{(0)})$, we denote its spherical-harmonic coefficients by
\begin{align*}
g^{(\ell,m)}
:=
\int_{\mathbb S^2}
g(R_0\widehat{x})\,\overline{Y_\ell^m(\widehat{x})}\,dS(\widehat{x}),
\qquad \ell\ge 0,\quad -\ell\le m\le \ell.
\end{align*}
When $l=m=0$, we denote $g^{(0,0)}$ by $g^0$. 
\end{lemma}

With Lemmas \ref{le:2} and \ref{le:3} in hand, we can now justify the
reduction to the classical linearized Rayleigh--Plesset equation and to its
Lamb-type generalizations.

\begin{theorem} \label{th:reduct_1}
Let $T,\sigma, \rho_\ell >0$ be fixed, and let $\mathcal K\subset\mathbb R^3$ be a compact set with Lipschitz boundary whose interior contains $\overline{\Omega_g^{(0)}}$. Suppose that 
\begin{align}\label{eq:1-prior}
I^{(0)} \in \widetilde {\mathcal X}^{(0)}_6, \quad  \mathcal I^{\perp} \in D((\mathcal A^{\perp})^3).  
\end{align}
Assume that $p_{\ell,0}$ and $q_{\ell,0}$ are compactly supported, and that $q_{g,0}$ and $p_{g,0}$ are constant functions in $\Omega^{(0)}_g$. 
\begin{enumerate}[(a)]
\item \label{a1}
Assume that 
$0< {\rho_g} \le \beta < \infty,$
where $\beta$ is independent of $c_g$ and $c_\ell$. For $l \ge 2$ and $ -l \le m \le l$, $\eta^{(l,m)}(t)$ admits the expansion for $t \in (0,T)$:
\begin{align}\label{eq:1-reduction}
\eta^{(l,m)}(t) = \eta^{(l,m)}_{\mathrm{dom}}(t) + \eta^{(l,m)}_{\mathrm{res}}(t),
\end{align}
where $\eta_{\mathrm{dom}}^{(l,m)}(t)$ solves
\begin{align} 
 &\left(\frac{\rho_\ell}{\ell+1} + \frac{\rho_g}{\ell}\right)\frac{d^2}{dt^2} \eta_{\mathrm{dom}}^{(l,m)}(t)  -\frac{\sigma}{R_{0}^3} (l+2)(1-l)\eta_{\mathrm{dom}}^{(l,m)}(t) = 0, \label{eq:2-reduction}\\
 & \eta_{\mathrm{dom}}^{(l,m)}(0) =  \eta^{(l,m)}_0,\quad  \frac{d}{dt} \eta_{\mathrm{dom}}^{(l,m)}(0) = \zeta^{(l,m)}_0, \label{eq:2-boundary}
\end{align}
and $\eta_{\mathrm{res}}^{(l,m)}(t)$ satisfies
\begin{align} \label{eq:3-reduction}
 \mathrm{sup}_{t\in [0,T]}\left|\eta^{(l,m)}_{\mathrm{res}}(t)\right| \le C \max\left(\frac{1}{c_{\ell}}, \frac{1}{c_{g}}\right), \quad \mathrm{as}\; c_g, c_l \rightarrow  \infty
\end{align}

Moreover, given any $\delta>0$, for the corresponding bulk fields, we have the expansion for $t \in (0,T)$:
\begin{align} \label {eq:10-LP}
  \left\langle p_\alpha(x,t), Y_l^{(m)}(\hat x)\right \rangle_{L^2(\mathbb S^2)}  Y_l^{(m)}(\hat x) = p^{(l,m),\rm{dom}}_\alpha(x) + p^{(l,m),\mathrm{res}}_{\alpha}(x,t), \quad \mathrm{in}\; \Omega_{\alpha}^{(0)}, \quad \alpha \in \{\ell,g\},
\end{align}
where,
$p_\ell^{(l,m)}$ and $p_g^{(l,m)}$ solve the harmonic transmission problem:
\begin{align*}
&\Delta p_\alpha^{(l,m),\rm{dom}}=0, \qquad \mathrm{in }\; \Omega^{(0)}_\alpha, \qquad \alpha \in \{\ell,g\},\\
&p_\ell^{(l,m),\rm{dom}}-p_g^{(l,m),\rm{dom}}= \frac{\sigma}{R_{0}^2} (l+2)(1-l)\eta^{(l,m)}(t)Y^m_l(\hat x),\qquad \mathrm{on}\; \Gamma,\;\\
&\frac{1}{\rho_g}\partial_\nu p_\ell^{(l,m),\rm{dom}} = \frac{1}{\rho_\ell} \partial_\nu p_g^{(l,m),\rm{dom}} = \frac{d^2}{dt^2} \eta^{(l,m)}(t)
\quad\mathrm {on }\;\Gamma,
\end{align*}
and $p^{(l,m),\mathrm{res}}_{\alpha}$ satisfies
\begin{align} 
 &\mathrm{sup}_{t\in [\delta,T]}\left[\left\|p^{(l,m),\mathrm{res}}_{\ell}(\cdot,t)\right\|_{H^1(\mathcal K\backslash\Omega^{(0)}_\ell)} + \left\|p^{(l,m),\mathrm{res}}_{g}(\cdot,t)\right\|_{H^1(\Omega^{(0)}_g)}\right] \notag \\
 &\le C \max\left(\frac{1}{c_{\ell}}, \frac{1}{c_{g}}\right), \quad \mathrm{as}\; c_g, c_l \rightarrow \infty. \notag 
\end{align}

\item \label{a2}
Assume that
$K_g\in (K_-,K_+)$, where $K_\pm$ are independent of $c_g $ and $c_l$.
$\eta^{(0)}(t)$ admits the expansion for $t \in (0,T)$:
\begin{align}\label{eq:5-reduction}
\eta^{(0)}(t) = \eta^{(0)}_{\mathrm{dom}}(t) + \eta^{(0)}_{\mathrm{res}}(t),
\end{align}
where $\eta_{\mathrm{dom}}^{(0)}(t)$ solves
\begin{align}
&\rho_\ell\frac{d^2}{dt^2}  \eta_{\mathrm{dom}}^{(0)}(t) + \left(\frac{3K_g}{R_0^2} - \frac{2 \sigma}{R_{0}^3} \right) \eta_{\mathrm{dom}}^{(0)}(t) = -3\frac{K_g}{R_0^2} \eta^{(0)}_0 -\frac{1}{R_0^2}\left(p_{g,0}^{(0)} + t q_{g,0}^{(0)}\right), \notag\\
&\eta_{\mathrm{dom}}^{(0)}(0) = \eta^{(0)}_0, \quad \frac{d}{dt} \eta_{\mathrm{dom}}^{(0)}(0) = 0. \notag
\end{align}
and $\eta_{\mathrm{res}}^{(0)}(t)$ satisfies
\begin{align*}
 \mathrm{sup}_{t\in [0,T]}\left|\eta^{(0)}_{\mathrm{res}}(t)\right| \le C \max\left(\frac{1}{c_{\ell}}, \frac{1}{c_{g}}\right) \quad \mathrm{as}\;\; c_g, c_\ell \rightarrow \infty
\end{align*}
Moreover, given any $\delta>0$, we have 
\begin{align}
&\left\langle p_g(x,t), Y_0^{(0)}(\hat x)\right\rangle_{L^2(\mathbb S^2)} = - \frac{3K_g}{R_0} (\eta^{(0)}(t)-\eta_0^{(0)}) -\left(p_{g,0}^{(0)} + t q_{g,0}^{(0)}\right) + p_g^{(0),\;\mathrm{rem}}, \label{eq:68}\\
& \left\langle p_\ell(x,t), Y_0^{(0)}(\hat x)\right\rangle_{L^2(\mathbb S^2)} = - \frac{\rho_\ell R_0^2}{r}\frac{d^2}{dt^2}\eta^{(0)}(t) + p_\ell^{(0),\;\mathrm{rem}}, \label{eq:69}
\end{align}
where
\begin{align*}
   \mathrm{sup}_{t\in [\delta,T], \; x\in \mathcal K\backslash \Omega_\ell^{(0)}}\left|p_\ell^{(0),\;\mathrm{rem}}(x,t)\right| + \mathrm{sup}_{t\in [0,T],\; x\in \Omega_g^{(0)}}\left|p_g^{(0),\;\mathrm{rem}}(x,t)\right|\le C \max\left(\frac{1}{c_{\ell}}, \frac{1}{c_{g}}\right),
\end{align*}
as $c_g, c_\ell \rightarrow \infty$.
\end{enumerate}
Here, $C$ is a positive constant independent of $c_g$ and $c_\ell$.
\end{theorem} 

\begin{proof}
\eqref{a1}
Let $U^{\perp}(t)$ denote the solution corresponding to the initial data $\mathcal I^\perp$.
Using \eqref{eq:1-prior} and statement \textup{(i)} of Theorem \ref{thm:wellposed-A}, we obtain that there exists a unique
\begin{align} \label{eq:1-reg}
U^{\perp}(t)\in C^3([0,\infty);\mathcal H_{\ge 1})\cap C^{2}([0,\infty);D(\mathcal A^{\perp})),
\end{align}
and that 
\begin{align}
    \frac{d^l}{dt} U^{\perp}(t) = (\mathcal A^{\perp})^l U^{\perp}(t), \quad l \in \{1,2,3\}. \notag
\end{align}
In conjunction with\eqref{eq:gronwall-step2}, \eqref{eq:1-prior} and the assumption that $q_{g,0}$ and $p_{g,0}$ are constant functions in $\Omega^{(0)}_g$, we have 
\begin{align*}
\sum^{2}_{l=0} \left\|\partial^l_t U^{\perp}(\cdot, t)\right\|_{\mathcal E} \le C .
\end{align*}
This, together with the Poincaré inequality implies 
\begin{align} \label{eq:2-pri}
\quad \mathrm{sup}_{t\in [{0,T}]}\left(\sum^3_{l=1}\|\partial^l _{t}\eta^{\perp}(\cdot,t)\|_{H^1({S^{(0)}})} + \left[\sum_{\alpha \in \{\ell,g\}}\sum^2_{l=1}\left\|\partial^l_t p^{\perp}(\cdot,t)\right\|_{H^{1/2}{(S^{(0)})}}\right]\right)& \le {C}.
\end{align}
Applying \eqref{eq:A-pt}, \eqref{eq:1-reg} and Lemma \ref{le:2}, we obtain
\begin{align}
& p^{\perp}_\ell(x,t) = \int_{S^{(0)}}\partial_\nu \frac{1}{4\pi|x-y|} p^{\perp}_\ell(y,t-c^{-1}_l|x-y|) - \nu(y)\cdot \frac{x-y}{4\pi c_\ell|x-y|^2}\partial_t p^{\perp}_\ell(y,t-c^{-1}_\ell |x-y|) \notag \\
& - \frac{1}{4\pi|x-y|}\rho_\ell \partial^{\perp}_{tt}\eta(y,t-c_{\ell}^{-1}|x-y|) dS(y) + p_\ell^{\mathrm{res}}(x,t), \quad x \in \Omega^{(0)}_\ell,\; t \in \mathbb R_+.\label{eq:3-LP}
\end{align}
Similarly, we have 
\begin{align}
& p^{\perp}_g(x,t) = \int_{S^{(0)}}\nu(y)\cdot \frac{x-y}{4\pi c_g |x-y|^2}\partial_tp^{\perp}_g(x,t-c^{-1}_g |x-y|) - \partial_\nu \frac{1}{4\pi|x-y|} p^{\perp}_g(x,t-c^{-1}_g|x-y|)\notag \\
& + \frac{1}{4\pi|x-y|}\rho_g \partial_{tt}\eta^{\perp}(x,t-c_g^{-1}|x-y|) dS(y) + p_g^{\mathrm{res}}(x,t), \quad  x\in \Omega^{(0)}_g,\; t \in \mathbb R_+. \label{eq:4-LP} 
\end{align}
Here,
\begin{align*}
p_\alpha^{\mathrm{res}}:= \frac{1}{c_\alpha^2}\left[I^{\rm {in}}_{\alpha}q^{\perp}_{\alpha,0} - \partial_t \mathcal I^{\rm {in}}_{\alpha}p^{\perp}_{\alpha,0}\right], \qquad \mathrm{for}\;\alpha \in \{\ell,g\}.
\end{align*}
Using \eqref{eq:1-reg} and \eqref{eq:2-pri}, and applying a Taylor expansion in time to the second term in the integrands of \eqref{eq:3-LP} and \eqref{eq:4-LP}, we obtain
\begin{align} \label{eq:5-LP}
p^{\perp}_{\ell}(x,t) &=  \int_{S^{(0)}}\partial_\nu \frac{1}{4\pi|x-y|} p^{\perp}_\ell(y,t) - \frac{1}{4\pi|x-y|} \rho_\ell \partial_{tt} \eta^{\perp}(y,t) dS(y) \notag \\
& + p_\ell^{\mathrm{res}}(x,t) + p^{\mathrm{rem}}_\ell(x,t), \quad x \in \Omega_\ell^{(0)},\; t\in \mathbb R_+,
\end{align}
and that 
\begin{align}\label{eq:6-LP}
p_{g}(x,t)& = -\int_{S^{(0)}}\partial_\nu \frac{1}{4\pi|x-y|} p^{\perp}_g(y,t) 
+ \frac{1}{4\pi|x-y|}\rho_g\partial_{tt}\eta^{\perp}(y,t) dS(y) \notag\\ 
& + p_g^{\mathrm{res}}(x,t)  + p^{\mathrm{rem}}_g(x,t), \quad x \in \Omega^{(0)}_g,\; t\in \mathbb R_+.
\end{align}
Here,
\begin{align}
&\mathrm{sup}_{t\in [0,T]}\left[\left\|p^{\mathrm{rem}}_g(\cdot, t)\right\|_{H^1(\Omega^{(0)}_g)} + \left\|p^{\mathrm{rem}}_\ell(\cdot, t)\right\|_{H^1(K\backslash\Omega^{(0)}_\ell)}\right] \le  \frac{C}{\min(c_g, c_\ell)}, \quad \mathrm{as}\;\; c_g, c_\ell \rightarrow \infty.\label{eq:7-LP}
\end{align}
Then, with the aid of the well-known identities 
\begin{align}
\int_{\mathbb S^2} \frac{1}{|\hat x- y|}(Y_\ell^{m}\big)(y) d S(y) &= \frac{4\pi}{2\ell+1}\,Y_\ell^{m}(\hat x),
\qquad \hat x\in \mathbb S^2, \notag \\
\mathrm{and}\;\operatorname{p.v.}\int_{\mathbb S^2}\partial_{\nu(y)}\frac{1}{|\hat x-y|}\,Y_\ell^{m}(y)\,dS(y)
&= -\frac{2\pi}{2\ell+1}\,Y_\ell^{m}(\hat x), 
\qquad \hat x\in \mathbb S^2. \notag
\end{align}
Taking the boundary limits on $\Gamma$ in \eqref{eq:5-LP} and \eqref{eq:6-LP}, we have that
\begin{align} \label{eq:8-LP}
\frac{1}2 p^{(l,m)}_\ell(t) =  -p^{(l,m)}_\ell(t)\frac 1{2(2\ell+1)}- \rho_\ell\frac{R_0}{2\ell+1} \frac{d^2}{dt^2} \eta^{(l,m)}(t) + \left (p_\ell^{\mathrm{res}}\right)^{(l,m)}(t) + \left (p_\ell^{\mathrm{rem}}\right)^{(l,m)}(t)
\end{align}
and that 
\begin{align}
\frac{1}2 p^{(l,m)}_g(t) = p^{(l,m)}_g(t)\frac 1{2(2\ell+1)} + \rho_g \frac{R_0}{2\ell+1} \frac{d^2}{dt^2}  \eta^{(l,m)}(t) + \left (p_g^{\mathrm{rem}}\right)^{(l,m)}(t) + \left (p_g^{\mathrm{rem}}\right)^{(l,m)}(t). \label{eq:9-LP}
\end{align}
Using \eqref{eq:harmonic}, \eqref{eq:jump-interface-final} and subtracting \eqref{eq:8-LP} and \eqref{eq:9-LP}, we have 
\begin{align*}
&R_0\left(\frac{\rho_\ell}{\ell+1} + \frac{\rho_g}{\ell}\right)\frac{d^2}{dt^2} \eta^{(l,m)}(t)  -\frac{\sigma}{R_{0}^2} (l+2)(1-l)\eta^{(l,m)}(t) \\
&= \left(p_g^{\mathrm{rem}}\right)^{(l,m)} - \left(p_\ell^{\mathrm{rem}}\right)^{(l,m)} +\left(p_g^{\mathrm{res}}\right)^{(l,m)} - \left(p_\ell^{\mathrm{res}}\right)^{(l,m)}, \\
&\eta^{(l,m)}(0) = \eta^{(l,m)}_0,\quad  \frac{d}{dt} \eta^{(l,m)}(0) =  \zeta^{(l,m)}_0 .
\end{align*}
Using the Duhamel principle representation of the above second-order ODE, \eqref{eq:7-LP} and Lemma \ref{le:3}, we readily obtain the expansion \eqref{eq:1-reduction}, where $\eta_{\mathrm{dom}}^{(l,m)}(t)$ solves \eqref{eq:2-reduction}--\eqref{eq:2-boundary} and the remainder term satisfies \eqref{eq:3-reduction}.

For the bulk fields, since for every $\delta>0$, $p_\ell^{\mathrm{res}}$ and $p_g^{\mathrm{res}}$ vanish uniformly for $ x \in \mathcal K,\; \delta\le t\le T$, using \eqref{eq:5-LP}, \eqref{eq:6-LP} and \eqref{eq:7-LP}, we obtain \eqref{eq:10-LP}.

\eqref{a2}
Let 
\begin{align*}
\left(p_{\ell}^{(0)}(r,t)Y^0_0(\hat x),p_{g}^{(0)}(r,t)Y^0_0(\hat x), q_\ell^{(0)}(r,t) Y^0_0(\hat x), q_{g}^{(0)}(r,t)Y^0_0(\hat x) ,\eta^{(0)}(t)Y^0_0(\hat x), \zeta^{(0)}(t)Y^0_0(\hat x)\right)
\end{align*}
denote the solution corresponding to the initial data $\mathcal I^{(0)}$. We note that 
\begin{align*}
\left(v^{(0)}_\ell(r,t),v_g^{(0)}(r,t), \eta^{(0)}\right) 
\end{align*}
solves equations \eqref{eq:t_1}--\eqref{eq:t_2}. Here, 
\begin{align} \label{eq:44}
v_\alpha^{(0)} = r p_\alpha^{(0)} \quad \mathrm{for}\; \alpha \in \{\ell,g\}.
\end{align}
Building upon \eqref{eq:1-prior} and the assumption that $q_{g,0}$ and $p_{g,0}$ are constant functions in $\Omega^{(0)}_g$, using Theorem \ref{thm:wellposed-B} and Remark \ref{rem:uniform-estimate-reduction}, we obtain that for some $\gamma >0 $
\begin{align}
&v^{(0)}_g(R_0,\cdot) \in H^6_\gamma(0,\infty), \;\; v^{(0)}_\ell(R_0,\cdot)\in H^6_\gamma(0,\infty), \notag \\
&\partial_t^{j-1} v^{(0)}_g 
\in L^2_\gamma(0,\infty;H^1(R_0,\infty),\;\; \partial_t^{j-1} v^{(0)}_\ell
\in L^2_\gamma(0,\infty;H^1(0,R_0),\qquad 1\le j\le 6, \notag
\end{align}
and that 
\begin{align}
& \sum^{4}_{j=0}\int^T_0|\partial^j_t v^{(0)}_g(R_0,t)|^2 dt + \int^T_0|\partial^j_t v^{(0)}_\ell(R_0,t)|^2 dt \le C. \label{eq:56} 
\end{align}
Here, $C$ is independent of $c_\ell$ and $c_g$.

We first consider the gas reduction. Set $V(t):=v^{(0)}_g(R_0,t)$.
We subtract the affine boundary part from the gas solution by writing
\begin{align*}
    v^{(0)}_g(r,t)=\frac r{R_0} V(t) + w_g(r,t),
\end{align*}
where
\begin{align*}
w_{g,tt}-c_g^2w_{g,rr} = -\frac r{R_0} V''(t),\\
w_g(0,t) = w_g(R,t) =0, \\
w_g(r,0) = \partial_t w_{g}(r,0)=0, \quad \mathrm{for}\;  0<r<R_0.
\end{align*}
Thus, 
\begin{align*}
F_g(t)
= \frac1{\rho_gR_0^2} \left(R_0\partial_rv_g(t,R) - V(t)
    \right) = \frac1{\rho_gR_0}\partial_rw_g(t,R_0).
\end{align*}
Expanding $w_g$ in a sine series,
\begin{align*}
    w_g(t,r)
    =
    \sum_{n\ge1}a_n(t)\sin\left(\frac{n\pi r}{R_0}\right),
\end{align*}
we obtain that
\begin{align*}
a_n''(t)+\omega_n^2a_n(t) =
    \frac{2(-1)^n}{n\pi}V''(t),
    \qquad
    \omega_n:=\frac{n\pi c_g}{R_0}.
\end{align*}
Since $a_n(0)=a_n'(0)=0$, it follows that
\begin{align*}
    a_n(t)
    =
    \frac{2(-1)^n}{n\pi}
    \int_0^t
    \frac{\sin\left(\omega_n(t-s)\right)}{\omega_n}
    V''(s)\,ds .
\end{align*}
Consequently,
\begin{align*}
    F_g(t) = \frac{2}{\rho_gR^2_0}
    \sum_{n\ge1}
    \int_0^t
    \frac{\sin\left(\omega_n(t-s)\right)}{\omega_n}
    V''(s)\,ds.
\end{align*}
Integrating by parts in time gives
\begin{align*}
  F_g(t) \frac{\rho_g R^2_0}{2}
    &=
    \frac{V''(t)-V''(0)\cos(\omega_n t)}{\omega_n^2}
    -
    \frac1{\omega_n^2}
    \int_0^t
    \cos\left(\omega_n(t-s)\right)V'''(s)\,ds .
\end{align*}
Using the well known identity $
\sum_{n\ge1}{n^{-2}\pi^{-2}}=1/6$, we obtain
\begin{align} \label{eq:61}
F_g(t) = \frac{1}{3K_g}V''(t)+\mathcal R_g(t),
\end{align}
where
\begin{align*}
    \mathcal R_g(t)
    &=
    -\frac{2}{K_g}
    \sum_{n\ge1}
    \frac1{n^2\pi^2}
    \left[
        V''(0)\cos(\omega_n t)
        +
        \int_0^t
        \cos\left(\omega_n(t-s)\right)V'''(s)\,ds
    \right].
\end{align*}
Integrating by parts once more and using the compatible condition $
V''(0)=0$, we obtain
\begin{align*}
 \int_0^t
        \cos\left(\omega_n(t-s)\right)V'''(s)\,ds=
    \frac{V'''(0)\sin(\omega_n t)}{\omega_n}
+
    \frac1{\omega_n}
    \int_0^t
    \sin\left(\omega_n(t-s)\right)V^{(4)}(s)\,ds .
\end{align*}
Hence
\begin{align*}
    |\mathcal R_g(t)|
    \le
    \frac{C}{K_gc_g}
    \left(|V'''(0)|
        +
        \|V^{(4)}\|_{L^1(0,T)}
    \right)
    \sum_{n\ge1}\frac1{n^3}.
\end{align*}
This, together with \eqref{eq:56} gives
\begin{align} \label{eq:57}
  \mathrm{sup}_{t\in [0,T]}  |\mathcal R_g(t)| \le \frac{C}{c_g} \quad \mathrm{as}\; c_g \rightarrow \infty.
\end{align}
Using \eqref{eq:t_2}, \eqref{eq:61}, \eqref{eq:57} gives 
\begin{align} \label{eq:64}
\frac{1}{3K_g}v^{(0)}_g(R_0,t) = -\eta^{(0)}(t) - \eta_0^{(0)} -\frac1 {3K_g}\left( p_{g,0}^{(0)} + t q_{g,0}^{(0)}\right) + \widetilde {\mathcal R}_g(t),
\end{align}
where 
\begin{align}
      \mathrm{sup}_{t\in [0,T]}  |\widetilde {\mathcal R}_g(t)| \le \frac{C}{c_g},\quad \mathrm{as}\;\; c_g \rightarrow \infty. \label{eq:66}
\end{align}

Next, we consider the high-speed asymptotics for the liquid component. Since
\begin{align*}
  \partial_t v^{(0)}_\ell(R_0,t) + c_\ell \partial_r v^{(0)}_\ell(R_0,t) =  \partial_t v^{(0)}_\ell(c_\ell t + R_0,0) + c_\ell \partial_r v^{(0)}_\ell(c_\ell t + R_0,0)
\end{align*}
Therefore,
\begin{align*}
\partial_r v^{(0)}_\ell(R_0,t) = \frac{1}{c_\ell}\left[\partial_t v^{(0)}_\ell(R_0,t) - q^{(0)}_{\ell,0}(c_\ell t + R_0,0) -c_\ell \partial_r p^{(0)}_{\ell,0}(c_\ell t + R_0,0)\right]=: \mathcal R_\ell(t)
\end{align*}
This, together with \eqref{eq:t_2} yields
\begin{align}\label{eq:60}
-\frac{1}{\rho_\ell R_0^2} v^{(0)}_\ell(R_0,t)= - \frac{d^2}{dt^2} \eta^{(0)}(t) - \frac{1}{\rho_\ell R_0} \mathcal R_\ell(t).
\end{align}

In conjunction with \eqref{eq:t_1}, \eqref{eq:64} and \eqref{eq:60} gives
\begin{align*}
&\rho_\ell R_0^2 \frac{d^2}{dt^2} \eta^{(0)}(t) + \left(3K_g  - \frac{2\sigma}{R_0}\right)\eta^{(0)}(t) = -3K_g \eta^{(0)}_0 -\left( p_{g,0}^{(0)} + t q_{g,0}^{(0)}\right) - R_0 \mathcal R_\ell(t) + 3K_g \widetilde {R}_g(t), \\
&\eta^{(0)}(0) = \eta^{(0)}_0, \quad \frac{d}{dt} \eta^{(0)}(0) = 0. \notag
\end{align*}
With the aid of \eqref{eq:56} and \eqref{eq:66}, proceeding in the derivation of \ref{eq:1-reduction}, we readily obtain that \eqref{eq:5-reduction}.

Finally, we consider the bulk fields.
Since for $0\le (r-R_0)<c_\ell t$, the solution of \eqref{eq:t_0} admits the representation 
\begin{align} \label{eq:67}
v^{(0)}_\ell(r,t)
&=
v^{(0)}_\ell\left(R_0, t-\frac{r-R_0}{c_\ell}\right)
+
\frac12\bigl[p^{(0)}_\ell(r + c_\ell t)-p^{(0)}_\ell(c_\ell t + 2R_0 - r)\bigr] \notag\\
&+ \frac{1}{2c_\ell}\int_{c_\ell t +2 R_0 - r}^{c_\ell t + r}q^{(0)}_\ell(s)\,ds.
\end{align}
Since the initial data are compactly supported, then for every fixed $M>0$ and every $\delta>0$, the initial-data terms vanish uniformly for $0\le x\le M,\; \delta\le t\le T$, provided $c_\ell$ is sufficiently large. This, together with \eqref{eq:67} and \eqref{eq:56} gives 
\begin{align} \label{eq:62}
v^{(0)}_\ell(r,t)
=
v^{(0)}_\ell\left(R_0, t\right) + O(c_\ell^{-1}), \quad \mathrm{as}\;\; c_\ell \rightarrow \infty,
\end{align}
on this region. Combining \eqref{eq:44}, \eqref{eq:60}, \eqref{eq:62} gives \eqref{eq:68}.

On the other hand, from the sine-series representation and the zero initial data for $w_g$, we have
\begin{align*}
    \sup_{t\in[0,T]}
    \|w_g(t,\cdot)\|_{H^1(0,R_0)}
    \le
    \frac{C_{R_0}}{c_g}\|V''\|_{L^1(0,T)}.
\end{align*}
Due to \eqref{eq:44}, we have
\begin{align*}
    p^{(0)}_g(t,r)=\frac{V(t)}{R_0}+\frac{w_g(t,r)}{r}.
\end{align*}
Using $w_g(t,0)=w_g(t,R_0)=0$ and the radial Hardy identity, we obtain
\begin{align*}
    \left\|\frac{w_g(t,\cdot)}{r}\right\|_{H^1(B_{R_0})}
    \le
    C\|w_g(t,\cdot)\|_{H^1(0,R_0)}.
\end{align*}
Therefore, with the aid of \eqref{eq:56},
\begin{align}
    \sup_{t\in[0,T]}
    \left\|
    p^{(0)}_g(t,\cdot)-\frac{V(t)}{R_0}
    \right\|_{H^1(B_{R_0})}
    \le
    \frac{C_{R_0}}{c_g}\|V''\|_{L^1(0,T)} \le \frac{C_{R_0}}{c_g}. \notag
\end{align}
This, together with \eqref{eq:64} gives \eqref{eq:69}.
\end{proof}

\subsection{Relation to the frozen transmission model}
We now consider the small-$\sigma$ reduction. The limiting model is the
standard frozen-interface acoustic transmission problem posed on the equilibrium
domains. In this model, the interface geometry is fixed and no interface
displacement is retained as an independent unknown.

Before proving the main theorem of this section, Theorem \ref{prop:sigma-to-zero}, we record
the following classical elliptic lifting lemma, whose proof appears in Appendix \ref{sec:B}.

\begin{lemma}\label{lem:lifting-jump}
Let $-1/2<s< 1/2$.
Then there exists a bounded linear lifting operator
\begin{align*}
\mathcal R:H^s(S^{(0)})\to 
H^{s+1/2}(\Omega^{(0)}_g)\times H^{s+1/2}(\Omega^{(0)}_\ell),
\qquad g\mapsto G =  \mathcal R g,
\end{align*}
such that $G$ satisfies
\begin{align}
&-\nabla\cdot\left(\frac{1}{\rho_\alpha}\nabla G_\alpha\right)=0
\qquad\mathrm{in}\; \Omega^{(0)}_\alpha,\qquad \alpha\in\{\ell,g\}, \notag\\
&G_g - G_\ell =g,\qquad
\frac{1}{\rho_g}\partial_{\nu}G_g = \frac{1}{\rho_\ell}\partial_{\nu_\ell}G_\ell \quad \mathrm{on }\;S^{(0)}, \notag
\end{align}
Moreover,
\begin{align} 
\|G_\ell\|_{H^{s+1/2}(\Omega^{(0)}_\ell)}
+
\|G_g\|_{H^{s+1/2}(\Omega^{(0)}_g)}
\le C \|g\|_{H^s(S^{(0)})}.\notag
\end{align}
\end{lemma}

\begin{theorem}\label{prop:sigma-to-zero} 
 Let $s \in (0,1/2)$. Assume that $T$, $K_\alpha$ and $\rho_\alpha$, $\alpha \in \{\ell,g\}$ are fixed, and that 
\begin{align}\label{eq:11-prior}
  \mathcal I^{(0)} \in \mathcal X^{(0)}_3, \quad  \mathcal I^\perp \in D((\mathcal A^{\perp})^3).  
\end{align}
Then there exists a
constant $C>0$, independent of $\sigma$, such that
\begin{align} \label{eq:sigma_1}
\sup_{t\in[0,T]}\left(\sum^{2}_{\alpha =1 }\left\|p_\alpha(\cdot,t)-p_\alpha^{\mathrm{fro}}(\cdot,t)\right\|_{H^{s+1/2}(\Omega^{(0)}_\alpha)}\right) \le  C \sigma^{1/4-s/2}, \qquad \mathrm{as}\; \sigma \rightarrow 0.
\end{align}
Here, $p^{\mathrm{fro}}$ denotes the solution of the following frozen-interface problem
\begin{align*}
&\frac{1}{K_\alpha}\partial_{tt} p_{\alpha}^{\mathrm{fro}}
-\frac{1}{\rho_\alpha}\Delta p_\alpha^{\mathrm{fro}} =0\qquad\mathrm{in}\; \Omega^{(0)}_\alpha \times \mathbb R_+,\qquad \alpha\in\{\ell,g\},\\
& p_{\ell}^{\mathrm{fro}} = p_{g}^{\mathrm{fro}},\qquad
\frac{1}{\rho_\ell}\partial_{\nu} p_{\ell}^{\mathrm{fro}} = \frac{1}{\rho_g}\partial_{\nu} p_{g}^{\mathrm{fro}} \quad \mathrm{on}\; S^{(0)} \times \mathbb R_+,
\end{align*}
with the initial data
\begin{align*}
p_{0}^{\mathrm{fro}} = p_{0} - \sigma \mathcal R\left(\Delta_{\partial {B_ {R_0}}} + \frac{2}{R_0^2}\right)\eta_0, \qquad  \partial_t p_{0}^{\mathrm{fro}} = q_{0} - \sigma \mathcal R \left(\Delta_{\partial {B_{R_0}}} + \frac{2}{R_0^2}\right) \zeta_0,
\end{align*}
where the lifting operator $\mathcal R$ is given in Lemma \ref{lem:lifting-jump}.
\end{theorem}

\begin{proof}
Throughout the proof, we assume that $\delta >0$ is sufficiently small.
Let $U^{\perp}(t)$ and 
\begin{align*}
\left(p_{\ell}^{(0)}(r,t)Y^0_0(\hat x),p_{g}^{(0)}(r,t)Y^0_0(\hat x), q_\ell^{(0)}(r,t) Y^0_0(\hat x), q_{g}^{(0)}(r,t)Y^0_0(\hat x),\eta^{(0)}(t)Y^0_0(\hat x), \zeta^{(0)}(t)Y^0_0(\hat x)\right)
\end{align*}
denote the solution corresponding to the initial data $\mathcal I^\perp$ and $\mathcal I^{(0)}$, respectively.

Proceeding as in the derivation of \eqref{eq:2-pri}, using Theorem \ref{thm:wellposed-A} and applying \eqref{eq:11-prior} gives 
\begin{align*}
\mathrm{sup}_{t\in [{0,T}]} \left(\sum^{2}_{l=1}\left\|\partial^l_t p^{\perp}(\cdot,t)\right\|_{H^1(\Omega^{(0)}_\ell\cup \Omega^{(0)}_g)} \right) \le C, \quad \mathrm{sup}_{t\in [{0,T}]} \left(\sum^{3}_{l=0}\left\|\partial^3 _{t}\eta^{\perp}(\cdot,t)\right\|_{H^1({S^{(0)}})}\right) \le \frac{C}{\sqrt \sigma}.
\end{align*}
In conjunction with the jump condition \eqref{eq:jump-interface-final}, we obtain that 
\begin{align} 
    &\mathrm{sup}_{t\in [{0,T}]} \left(\sum^{2}_{l=1}\left\|\partial^l_t [p^{\perp}](\cdot,t)\right\|_{H^{-1}(S^{(0)})} \right) \le C\sqrt \sigma, \label{eq:65}\\&\mathrm{sup}_{t\in [{0,T}]} \left(\sum^{2}_{l=1}\left\|\partial^l_t [p^{\perp}](\cdot,t)\right\|_{H^{1/2}(S^{(0)})} \right) \le C. \label{eq:19}
\end{align}
Furthermore, using \eqref{eq:jump-interface-final}, \eqref{eq:11-prior} and Theorem \ref{thm:wellposed-B}, we have
 \begin{align*}
\mathrm{sup}_{t\in [{0,T}]} \left(\sum^{2}_{l=1}\left|\partial^l_t [p^{(0)}](R_0,t)\right|\right) \le  C \sigma \mathrm{sup}_{t\in [{0,T}]} \left(\sum^{2}_{l=1} |\eta^{(0)}(t)|\right)\le C\sigma.
 \end{align*}
This, together with \eqref{eq:65}, \eqref{eq:19} and interpolation
\begin{align} \label{eq:interpolation}
\|g(t)\|_{H^{\theta-1/2}(S^{(0)})}
\le
C
\|g(t)\|_{H^{-1}(S^{(0)})}^{1-\theta}
\|g(t)\|_{H^{1/2}(S^{(0)})}^{\theta}, \quad \theta \in (0,1)
\end{align}
gives
\begin{align*}
\mathrm{sup}_{t\in [{0,T}]} \left(\sum^{2}_{l=1}\left\|\partial^l_t [p](\cdot,t)\right\|_{H^s(S^{(0)})} \right) \le C \sigma^{1/4-s/2}, \quad s = \theta - 1/2\; \textrm{with}\; \theta \in (1/2,1).
\end{align*}
Let $p^{\mathrm{dif}}:= \mathcal R [p]$. Since $\mathcal R$ is linear and time-independent, we have that 
\begin{align*}
\partial_t p^{\mathrm{dif}} = \mathcal R \partial_t[p],\qquad \partial_{tt} p^{\mathrm{dif}} = \mathcal R \partial_{tt}[p],
\end{align*}
and that 
\begin{align}
\left\|\partial^l_t p^{\mathrm{dif}}(\cdot,t)\right\|_{H^{s+1/2}(\Omega_\ell \cap \Omega_g)}
\le C\|\partial^l_t [p](\cdot,t)\|_{H^s(\Gamma)}, \quad t\in [0,T], \;\; l\in\{1,2\}.\notag
\end{align}
Combining this with \eqref{eq:interpolation} gives 
\begin{align} \label{eq:diff}
\mathrm{sup}_{t\in [{0,T}]} \left(\sum_{l=0}^2\left\|\partial^l_t p^{\mathrm{dif}}(\cdot,t)\right\|_{H^{s+1/2}(\Omega_\ell \cap \Omega_g)}\right) \le C \sigma^{1/4-s/2}.
\end{align}
Furthermore, $p - p^{\mathrm{dif}} - p^{\mathrm{fro}}$ solves the frozen-interface problem with the right-hand source $-\partial_{tt}p^{\mathrm{dif}}$ together with the vanishing initial data. Together with the well-posedness estimate for the frozen-interface transmission
model with vanishing initial data, estimate \eqref{eq:diff} implies \eqref{eq:sigma_1}.

\end{proof}

\section{Resonance regimes}
\label{sec:resonance-connection}

In this section, we discuss resonances near the real axis and explain how the present
model is related to several classical resonance mechanisms arising in bubble dynamics and high-contrast transmission problems.

The full linearized model keeps the displacement of the interface as an explicit unknown. In this formulation, the bulk acoustic fields are coupled to the interface motion through a dynamic boundary condition. This allows one to
recover, at the linearized level, the classical Minnaert-type breathing mode as well as shape oscillations of Lamb type.

Another standard reduction arises from the high-contrast nature of gas bubbles
in liquids. When the interface motion is neglected, the bubble is modeled as a
fixed high-contrast inclusion, leading to a frozen-interface transmission problem. In the regime, the corresponding scattering resonances are closely related to the interior
Neumann spectrum. The first subwavelength resonance reproduces the
Minnaert-type mode, whereas higher resonances may be interpreted as
Fabry--P\'erot-type modes.

The purpose of this section is to make this connection explicit and to indicate how the different resonance regimes arise from the same linearized fluid--bubble formulation.

We now pass to the frequency domain. For simplicity, we again consider a single
spherical bubble $B_{R_0}$ of a radius $R_0$ centered at $0$ and we use the notation introduced at the beginning of Section \ref{sec:reduced-models}.
We focus on the time-harmonic regime $e^{-i\omega t}$. A resonance condition is a complex frequency $\omega$ for which the corresponding
homogeneous problem admits a nontrivial solution $(u_\ell, u_{g}, \hat \eta)$, namely
\begin{align}
\Delta u_\alpha + k_\ell^2 u_\alpha = 0\quad &\text{in } \Omega^{(0)}_\alpha, \quad \alpha \in \{\ell,g\},  \label{eq:time_harmonic_liquid}\\
\frac{1}{\rho_\ell}\partial_n u_\ell
= \frac{1}{\rho_g}\partial_n u_g = \omega^2 \hat\eta, \quad & \text{on}\; S^{(0)},
\label{eq:10-kinematic-freq}\\
u_\ell - u_g
= {\sigma}\left(\Delta_{S^{(0)}}  + \frac{2}{R_0^2}\right)\hat\eta, \quad &\text{on}\;S^{(0)},
\label{eq:10-dynamic-freq}\\
u_\ell\quad \mathrm{is}\; k_\ell\; \mathrm{outgoing}. \notag
\end{align}
Here, $k_\alpha^2 = {\omega^2}/{c_\alpha^2}$ for $\alpha\in\{\ell,g\}$.
To analyze this resonance problem, we use the spherical symmetry of the
configuration. We write every point on the equilibrium interface
$S$ as
\begin{align*}
x=R_0\hat x,\qquad 
\hat x=(\sin\theta\cos\varphi,\sin\theta\sin\varphi,\cos\theta)\in\mathbb S^2,
\end{align*}
where $(\theta,\varphi)$ are the usual spherical coordinates. We expand the
interface displacement and the boundary traces in spherical harmonics $Y_l^m$ on $\mathbb S^2$. Namely,
\begin{align*}
\widehat{\eta}(R_0\hat x)
=
\sum_{l=0}^{\infty}\sum_{m=-l}^{l}
\eta^{(l,m)}Y_l^m(\hat x),
\end{align*}
and similarly,
\begin{align*}
u_\ell(R_0\hat x)
= \sum_{l=0}^{\infty}\sum_{m=-l}^{l}
u^{(l,m)}_{\ell}Y_l^m(\hat x),
\qquad
u_g(R_0\hat x)
=
\sum_{l=0}^{\infty}\sum_{m=-l}^{l}
u^{(l,m)}_{g}Y_l^m(\hat x).
\end{align*}
Since $u_\ell$ is the outgoing solution of the exterior Helmholtz equation, while $u_g$ solves the interior Helmholtz equation; see \eqref{eq:time_harmonic_liquid}, the radial dependence of each
$(l,m)$-mode is given by 
\begin{align*}
u^{(l,m)}_\ell(\theta,\varphi) = A^{(l,m)} h_l^{(1)}(k_\ell R_0) Y_l^m(\theta,\varphi),
\qquad
u^{(l,m)}_g(\theta,\varphi) =  C^{(l,m)}  j_l(k_g R_0) Y_l^m(\theta,\varphi),
\end{align*}
where $j_l$ denotes the spherical Bessel function of the first kind and
$h_l^{(1)}$ denotes the spherical Hankel function of the first kind.
Substituting these expansions into \eqref{eq:10-kinematic-freq}--\eqref{eq:10-dynamic-freq} and using the orthogonality of
the spherical harmonics, we obtain,
for each $(l,m)$, a $3\times 3$ linear system for $\left(A^{(l,m)}, C^{(l,m)}, \eta^{(l,m)}\right)$:
\begin{align}
 A^{(l,m)} h_l^{(1)}(k_\ell R_0)
-  C^{(l,m)} j_l(k_g R_0)
- \frac{\sigma}{R_0^2}(l+2)(1-l) \eta^{(l,m)}  &= 0,
\label{eq:10-BC1-lm}\\
\frac{k_\ell}{\rho_\ell} A^{(l,m)} {h_l^{(1)}}'(k_\ell R_0)
- \omega^2  \eta^{(l,m)}  &= 0,
\label{eq:10-BC2-lm}\\
\frac{k_g}{\rho_g}  C^{(l,m)}  j_l'(k_g R_0)
- \omega^2 \eta^{(l,m)} &= 0.
\label{eq:10-BC3-lm}
\end{align}
Building upon the above equations \eqref{eq:10-BC1-lm}, \eqref{eq:10-BC2-lm} and \eqref{eq:10-BC3-lm}, we can readily obtain that the existence of nontrivial $\omega$ if and only if $\omega$ is a non-zero solution of the following equation 
\begin{align}
\omega^2 h_l^{(1)}(k_\ell R_0){{j'_l}(k_g R_0)} - \omega^2 \frac{k_\ell}{\rho_\ell} \frac{\rho_g}{k_g}{{h_l^{(1)}}'(k_\ell R_0) j_l(k_g R_0)} =
\frac{\sigma}{R^2_0}(l+2)(1-l)\frac{k_\ell}{\rho_\ell}{h_l^{(1)}}'(k_\ell R_0){{j'_l}(k_g R_0)}.\notag
\end{align}
We note that ${h_\ell^{(1)}}'(k_\ell R_0) \ne 0$ when $k_\ell$ is near the real axis, we can write the above equation as 
\begin{align}\label{eq:10-dispersion-full}
    \mathcal J^{(l)}(\omega) = 0.
\end{align}
Here,  
\begin{align*}
&\mathcal J^{(l)}(\omega):=\omega^2 {{j'_l}(k_g R_0)} \frac{1}{k_\ell}\frac{h_l^{(1)}(k_\ell R_0)}{{h_l^{(1)}}'(k_\ell R_0)} - \omega^2\frac{1}{\rho_\ell}\frac{\rho_g}{k_g}{j_l(k_g R_0)} + (l+2)(l-1)\frac{\sigma}{R^2_0}\frac1{\rho_\ell} {{j'_l}(k_gR_0)}.
\end{align*}

Now we state the main theorem of this section.

\begin{theorem} \label{prop:resonance}
Let $\sigma, \rho_\ell,\; R_0 > 0$ be fixed. The following arguments hold true. 
\begin{enumerate}[(a)]
\item \label{b1}
Assume that the physical parameters satisfy
\begin{align*}
0< {\rho_g} \le \beta < \infty,
\end{align*}
where $\beta$ is independent of $c_g$ and $c_\ell$.
For each mode $l\ge 2$, there exist Lamb-type resonances
satisfying the asymptotic expansion below.
\begin{align} \label{eq:asym_l}
    w^{(l),\pm}_L = \pm \sqrt{\frac{(l+2)(l-1){\sigma}}{\left(\frac{\rho_\ell}{l+1} + \frac{\rho_g}{l}\right)R^3_0}} + O\left(\max\left(c^{-1}_g, c_l^{-1}\right)\right), \;\;\textrm{as}\; c_g, c_\ell \rightarrow \infty.
\end{align}
\item \label{b2}
Assume that
$K_g\in (K_-,K_+)$, where $K_\pm$ are independent of $c_g $ and $c_l$. There exist Minnaert resonances, whose asymptotic behavior are given by
\begin{align} \label{eq:asym_M}
\omega^{\pm}_M =  \pm \sqrt{\frac{3 K_g}{\rho_l R_0^2} - \frac{2\sigma}{\rho_l R^3_0}} + O\left(\max\left(c^{-1}_g, c^{-1}_l\right)\right), \quad \textrm{as}\; c_g, c_\ell \rightarrow \infty.
\end{align}
\item \label{b3}
Assume that $c_\ell \in (c_-, c_+)$ and  $c_g\in (c_-, c_+)$, where $c_\pm$ are independent of $\rho_g$. For each model $l\ge 1$, there exist Fabry--P\'erot-type resonances satisfying the  asymptotic expansion:
\begin{align} \label{eq:35}
  w^{(l,\pm)}_F:= \pm \lambda^{(l)} + O(\rho_g)\quad \mathrm{as}\; \rho_g \rightarrow 0,
\end{align}
where $\left(\lambda^{(l)}\right)^2/c_g^2$ denotes the Neumann eigenvalue of $-\Delta$ in $B_{R_0}$ associated with the $l$-th modes.
\end{enumerate}
Here and in the above asymptotic formulas, the constants implicit in the $O(\cdot)$-terms are independent of the limiting parameters in the corresponding regime.
\end{theorem}

\begin{proof}
\eqref{b1}
It is well known that
\begin{align}
&{t}\frac{j_0(t)}{j'_0(t)} = -{3} + O(t^2), \quad \label{eq:29}\\
& \frac{1}t \frac{j_l(t)}{j'_l(t)} = \frac{1}{l} + O\left({t^2}\right), \quad l \ge 1,\; \mathrm{as}\; t\rightarrow 0 \notag
\end{align}
and 
\begin{align}\label{eq:27}
& \frac{1}t \frac{h^{(l)}_0(t)}{{h^{(l)}_0}'(t)} =
\begin{cases}
-1 + O(t), \\
-\frac{1}{l+1} + O(t^2), \quad \mathrm{for}\;  l \ge 1, \; \text{as}\; t\rightarrow 0,
\end{cases}
\end{align}

For each $l \ge 1$, using the above asymptotic identities and \eqref{eq:10-dispersion-full}, we obtain
\begin{align}\label{asym_1}
\omega^2\left(\frac{1}{l+1} + \frac{\rho_g}{\rho_\ell} \frac{1}{l}\right) + \omega^2 O(\max(k_g^2, k_l^2)) = 
(l+2)(l-1)\frac{\sigma}{R^3_0}\frac1{\rho_\ell} \quad l \ge 1,  \;   \textrm{as}\; k_g, k_\ell\rightarrow 0.
\end{align}
Let
\begin{align*}
\delta:=\max(k_g^2,k_\ell^2),\qquad w_{\mathrm{dom}}^{(l)}:=\sqrt{\frac{(l+2)(l-1){\sigma}}{\left(\frac{\rho_\ell}{l+1} + \frac{\rho_g}{l}\right)R^3_0}}.
\end{align*}
By \eqref{asym_1}, in a fixed neighbourhood of the two unperturbed roots
$\omega_{\mathrm{dom}}^{(l,\pm)} =
\pm \omega_{\mathrm{dom}}^{(l)}$, $J^{(l)}(\omega)$ can be written as
\begin{align*}
J^{(l)}(\omega)
= F^{(l)}(\omega) +
F^{(l)}_{\mathrm{rem}}(\omega),
\end{align*}
where
\begin{align*}
 F^{(l)}(\omega):= \omega^2\left(\frac{1}{l+1} + \frac{\rho_g}{\rho_\ell} \frac{1}{l}\right) - (l+2)(l-1)\frac{\sigma}{R^3_0}\frac1{\rho_\ell} \quad \mathrm{and}\quad  F^{(l)}_{\mathrm{rem}}(\omega):= \mathcal J(\omega) - F^{(0)}(\omega).
\end{align*}
with 
\begin{align*}
|F^{(l)}_{\mathrm{rem}}(\omega)|
\le
C_0\delta,
\end{align*}
uniformly for $\omega$ in this neighbourhood, as $k_g,k_\ell\to0$. Here, $C_0>0$ is independent of $\delta$.
Since the zeros $\omega_{\mathrm{dom}}^{(l,\pm)}$ of $F^{(l)}$ are simple, there exist constants $C_1>0$ and $r_0>0$, independent of $\delta$, such that
\begin{align*}
\left|F^{(l)}(\omega)\right|
\ge
C_1\left|\omega-\omega_{\mathrm{dom}}^{(l,\pm)}\right|\quad
\mathrm{whenever}\;\;
\left|\omega-\omega_{\mathrm{dom}}^{(l,\pm)}\right|\le r_0 .
\end{align*}
Choose $M>0$ so large that $ C_0 < C_1 M.$
For $\delta>0$ sufficiently small, define the circles
\begin{align*}
\Gamma^{\pm}_{l,\delta}
:= \left\{
\omega\in\mathbb C:
\left|\omega-\omega_{\mathrm{dom}}^{(l,\pm)}\right| = M\delta \right\}.
\end{align*}
On $\Gamma^{\pm}_{l,\delta}$, we have
\begin{align*}
\left|F^{(l)}(\omega)\right|
\ge cM\delta > C\delta \ge
|F^{(l)}_{\mathrm{rem}}(\omega)|.
\end{align*}
Therefore, by Rouch\'e's theorem, $\mathcal J^{(l)}$ and $F^{(l)}$ have the same number of zeros inside each circle $\Gamma^{\pm}_{l,\delta}$. Since $F^{(l)}$ has exactly one zero inside $\Gamma^{\pm}_{l,\delta}$, namely $w^{(l),\pm}_L$, $\mathcal J^{(l)}$ also has exactly one zero $w^{(l),\pm}_L$ inside $\Gamma^{\pm}_{l,\delta}$. Consequently,
\begin{align*}
\left|w^{(l),\pm}_L - \omega_{\mathrm{dom}}^{(l,\pm)}\right|
\le M\max(k_g^2,k_\ell^2).
\end{align*}
This proves \eqref{eq:asym_l}.

\eqref{b2} For the case of monopole mode ($l = 0$), using \eqref{eq:10-dispersion-full}, \eqref{eq:29} and \eqref{eq:27}, we have 
\begin{align}\label{eq:33} 
\omega^2(1 + O(\max(k_g, k_l)) = \frac{\rho_g}{\rho_\ell} \frac{3c_g^2}{R^2_0}(1+ O(\max(k_g^2, k_l^2))) - \frac{2\sigma}{\rho_l R^3_0},\; \quad \textrm{as}\; k_g, k_\ell \rightarrow 0.
\end{align}
Let 
\begin{align*}
 F^{(0)}(\omega):= \omega^2 - \left(\frac{\rho_g}{\rho_\ell}\frac{3c_g^2}{R^2_0} -\frac{2\sigma}{R^3_0}\frac1{\rho_\ell}\right). 
\end{align*}
Then, the proof follows by the same argument as that of \eqref{eq:asym_l}: using the expansion \eqref{eq:33} and  applying Rouch\'e's theorem on a shrinking contour around the roots of $F^{(0)}(\omega)$, we readily obtain \eqref{eq:asym_M}.

\eqref{b3} 
We note that ${h_\ell^{(1)}}(k_\ell R_0) \ne 0$  near the real-axis.
Then, we can rewrite 
\eqref{eq:10-dispersion-full} as
\begin{align} \label{eq:34}
     \omega^2 {{j'_l}(k_g R_0)} - \omega^2\frac{1}{\rho_\ell}\frac {\rho_g}{k_g}{j_l(k_g R_0)} {k_\ell}\frac {{h_l^{(1)}}'(k_\ell R_0)}{h_l^{(1)}(k_\ell R_0)} = -(l+2)(l-1)\frac{\sigma}{R^2_0}\frac1{\rho_\ell} {{j'_l}(k_g R_0)} {k_\ell}\frac {{h_l^{(1)}}'(k_\ell R_0)}{h_l^{(1)}(k_\ell R_0)}.  
\end{align}
Since
\begin{align*}
\operatorname{Im}\left(
{h_l^{(1)'}(r)}/{h_l^{(1)}(r)}
\right) > 0
\qquad \text{for every real } r\neq 0,
\end{align*}
the ratio 
$h_l^{(1)'}(r)/h_l^{(1)}(r)$ is not real on the real axis.
Therefore, for each of the points $\pm\lambda^{(l)}$, one can choose a
punctured neighborhood in which $g_0(\omega)  \ne 0$. Here, $g_0(\omega)$ is defined as follows:
\begin{align*}
g_0(\omega):=\left(\omega^2 + (l+2)(l-1)\frac{\sigma}{R^2_0}\frac1{\rho_\ell}{k_\ell}\frac {{h_l^{(1)}}'(k_\ell R_0)}{h_l^{(1)}(k_\ell R_0)} \right){{j'_l}(k_g R_0)}.
\end{align*}
Then, we can rewrite \eqref{eq:34} as
\begin{align*}
 {{j'_l}(k_g R_0)} - \rho_g g_1(\omega) = 0, 
\end{align*}
where
\begin{align*}
   g_1(\omega):= \omega^2\frac{1}{\rho_\ell}\frac {1}{k_g}{j_l(k_g R_0)} {k_\ell}\frac {{h_l^{(1)}}'(k_\ell R_0)}{h_l^{(1)}(k_\ell R_0) g_0(\omega)}.
\end{align*}
Then, proceeding as in the derivation of \eqref{eq:asym_l}: applying Rouch\'e's theorem on a shrinking contour around the roots of ${{j'_l}(k_g R_0)})$, we readily obtain \eqref{eq:35}.
\end{proof}

\section*{Acknowledgment}
This work is supported by the Austrian Science Fund (FWF) grant P: 36942. 

\begin{appendices}
    
\section {Proofs of Lemma \ref{le:multiplier} and inequalities (\ref{eq:30})--(\ref{eq:31})}
\label{sec:A}

\begin{proof}[Proof of Lemma \ref{le:multiplier}]
By a straightforward calculation, we have 
\begin{align*}
    sC_\ell(s) = \frac{\rho_\ell R_0^2s}{1+{sR_0}/{c_\ell}}.
\end{align*}
Writing $\alpha_\ell:= {R_0}/{c_\ell}$, we obtain
\begin{align*}
\operatorname{Re}\left(sC_\ell(s)\right) =
 \rho_\ell R_0^2 \frac{\gamma + \alpha_\ell |s|^2}
{|1+\alpha_\ell s|^2}.
\end{align*}
Since 
\begin{align*}
 \frac{\gamma + \alpha_\ell |s|^2}
{|1 + \alpha_\ell s|^2} -  \frac{\gamma}
{1 + \alpha \gamma} = \frac{\alpha \xi^2}{|1+\alpha_\ell s|^2 (1+\alpha \gamma)}, 
\end{align*}
Hence
\begin{align}\label{eq:45}
\operatorname{Re}\left(sC_\ell(s)\right)
\ge \rho_{\ell}R_0^2
    \frac{\gamma}{1 + R_0\gamma/c_\ell} \ge \rho^*_{\ell}R_0^2
    \frac{\gamma}{1 + R_0\gamma/c^*_\ell} =: \widetilde m_\gamma.
\end{align}
Furthermore, it is known that 
\begin{align}\label{eq:46}
\operatorname{Re}\left(sC_g(s)\right)\ge0,
    \qquad
    \operatorname{Re}s>0.
\end{align}
Furthermore, it can be seen that there exists $\gamma^*$ such that $\gamma \ge \gamma_*$ 
\begin{align*}
\left|\frac{2\sigma}{R_0 s}\right| \le \frac 1 2 \widetilde m_\gamma.
\end{align*}
Combining this with \eqref{eq:45} and \eqref{eq:46} gives 
\begin{align}
 \operatorname{Re} (M(s)) \ge \frac{1}2 \widetilde m_\gamma, \quad \mathrm{for}\; s \in \{z \in \mathbb C: \mathrm{Re}(s) = \gamma \}\; \mathrm{with}\; \gamma > \gamma^*, \notag
\end{align}
which implies \eqref{eq:estimate} holds.

    
\end{proof}

\begin{proof}[Proof of \eqref{eq:30}]
It is known that 
\begin{align*}
\widehat v_\ell^{\,f}(r,s)
= \frac1{2c_\ell s}
    \int_{R_0}^{+\infty}
    \left(
        e^{-\lambda_\ell|r-y|}
        -
        e^{-\lambda_\ell(r+y-2R_0)}
    \right)
    \left(
        s V_\ell^{[0]}(y) + V_\ell^{[1]}(y)
    \right)\,dy, \quad s = c_\ell\lambda_\ell.
\end{align*}
Taking the $r$-derivative at $r = R_0$, we obtain
\begin{align*}
    \partial_r\widehat v_\ell^{\,f}(R_0,s)
    =
    \frac1{c_\ell^2}
    \int_{R_0}^{+\infty}
    e^{-\lambda_\ell (y-R_0)}
    \left(
        s V_\ell^{[0]}(y) + V_\ell^{[1]}(y)
    \right)\,dy .
\end{align*}
Integration by parts gives
\begin{align*}
    \partial_r\widehat v_\ell^{\,f}(R_0,s)
    &=
    \frac{V_\ell^{[0]}(R_0)}{c_\ell}
    +
    \frac1{c_\ell}
    \int_{R_0}^{+\infty} e^{-\lambda_\ell (y-R_0)}V_\ell^{'[0]}(y)\,dy
    +
    \frac1{c_\ell^2}
    \int_{R_0}^{+\infty} e^{-\lambda_\ell (y-R_0)} V_\ell^{[1]}(y)\,dy .
\end{align*}
Therefore
\begin{align*}
C_\ell(s)\widehat H_\ell(s)
=
\frac{\rho_\ell R_0^2}{1 + sR_0 / c_\ell}
\frac1{\rho_\ell R_0}
\partial_r\widehat v_\ell^{\,f}(s,0) =
\frac{R_0}{1 + \lambda_\ell R_0}
\partial_r\widehat v_\ell^{\,f}(s,0).
\end{align*}
We now estimate the three terms separately. First,
\begin{align*}
    \left\|
        \frac{R_0}{1 + \lambda_\ell R_0}
        \frac{V_\ell^{[0]}(R_0)}{c_\ell}
    \right\|_{L_\xi^2(\mathbb R)}
    &\le
    C_{\gamma, R_0} c_\ell^{-1/2}\left|V_\ell^{[0]}(R_0)\right|.
\end{align*}
By the one-dimensional trace theorem,
\begin{align*}
\left|V_\ell^{[0]}(R_0)\right| \le C \left\|V_\ell^{[0]}\right\|_{H^1(R_0,\infty)}.
\end{align*}
Next, for any $f\in L^2(R_0,\infty)$, the change of variables
$\eta=\xi/c_\ell$ and Plancherel's theorem imply
\begin{align*}
    \left\|
        \int_{R_0}^{+\infty} e^{-(\gamma+i\xi)y/c_\ell}f(y)\,dy
    \right\|_{L_\xi^2(\mathbb R)}
    \le
    C_{\gamma} c_\ell^{1/2}\left\|f\right\|_{L^2(R_0,\infty)}.
\end{align*}
Hence
\begin{align*}
    \left\|
    \frac{R_0}{1+\lambda_\ell R_0}
        \frac1{c_\ell}
        \int_{R_0}^{+\infty} e^{-\lambda_\ell y}V_\ell^{'[0]}(y)\,dy
    \right\|_{L_\xi^2(\mathbb R)}
    &\le
    C_{\gamma, R_0} c_\ell^{-1/2}\left\|V_\ell^{'[0]}\right\|_{L^2(R_0,\infty)},
\end{align*}
and
\begin{align*}
    \left\|
    \frac{R_0}{1+\lambda_\ell R_0}
        \frac1{c_\ell^2}
        \int_{R_0}^{+\infty} e^{-\lambda_\ell y}V_\ell^{[1]}(y)\,dy
    \right\|_{L_\xi^2(\mathbb R)}
    &\le
    C_{\gamma, R_0} c_\ell^{-3/2}\left\|V_\ell^{[1]}\right\|_{L^2(R_0,\infty)}.
\end{align*}
Combining these estimates gives \eqref{eq:30}.
\end{proof}

\begin{proof}[Proof of \eqref{eq:31}]
Define
\begin{align*}
 \widehat w_g(r,s)
    := \widehat v_g^{\,f}(r,s)- C_g(s)\widehat H_g(s) \frac{\sinh(s r/c_g)}{\sinh(s R_0/c_g)}, \quad  0 \le r \le R_0.
\end{align*}
By a straightforward calculation, we have 
\begin{align} \label{eq:42}
\begin{cases}
    s^2\widehat w_g-c_g^2\partial_{rr}\widehat w_g
    =
    sV_g^{[0]}+V_g^{[1]},
        &0<r<R_0,\\
    \widehat w_g(0,s)=0,\\
    R_0\partial_r\widehat w_g(R_0,s)-\widehat w_g(R_0,s)=0.
\end{cases}
\end{align}
Since 
\begin{align} \label{eq:18}
\widehat w_g(R_0,s) = -C_g(s)\widehat H_g(s),
\end{align}
in what follows, we derive estimates for $\widehat w_g(R_0,s)$ from system \eqref{eq:42}.

Let $V:=\{u\in H^1(0,R_0):u(0)=0\}$,
and define
\begin{align}\label{eq:14}
a_{R_0}(u,v) :=
    \int_0^{R_0} u'(r)\overline{v'(r)}\,dr
    -
    \frac1 {R_0} u(R_0)\overline{v(R_0)},
    \qquad u,v\in V.
\end{align}
Let $A_{R_0}$ be the self-adjoint operator in $L^2(0,R_0)$ associated with
$a_{R_0}$.  Since 
\begin{align*}
|u(R_0)|^2 \le
R_0 \int_0^{R_0} |u'(r)|^2\,dr.
\end{align*}
Hence
\begin{align*}
 a_{R_0}(u,u)
 =
\int_0^{R_0} |u'(r)|^2\,dr
    -
    \frac1{R_0}|u(R_0)|^2
    \ge 0.
\end{align*}
Equality holds if and only if $u'$ is constant, and since $u(0)=0$,
this means
\begin{align*}
\ker A_{R_0}= \operatorname{span}\{r\}.
\end{align*}
Let $\phi_0$ be the $L^2(0,R)$-normalized function proportional to $r$. We decompose
\begin{align*}
    \widehat w_g(R_0,s)
    =
    \widehat w^0_g(R_0,s)\phi_0 + \widehat w_g^\perp(R_0,s),
\end{align*}
where $\widehat w_g^\perp(R_0,s)\perp \phi_0 $ in $L^2(0,R)$. Similarly, write
\begin{align*}
V_g^{[0]} = g_0\phi_0+g_\perp, \qquad V_g^{[1]}=\widetilde g_0\phi_0 + \widetilde g_\perp.
\end{align*}

The zero mode satisfies
\begin{align*}
    s^2\widehat w^0_g(R_0,s) = s g_0 + \widetilde g_0.
\end{align*}
Since \(s=\gamma+i\xi\) and \(\gamma>0\),
\begin{align*}
    \left\|
        \widehat w^0_g(R_0,\gamma+i\xi)
    \right\|_{L^2_\xi}
    \le C_\gamma
    \left( |g_0|+|\widetilde g_0| \right).
\end{align*}
Moreover,
\begin{align*}
    |g_0|\le \left\|V_g^{[1]}\right\|_{L^2(0,R_0)}, \quad    |\widetilde g_0|
    \le \left\|V_g^{[0]}\right\|_{L^2(0,R_0)}
    \le C_{R_0}\left\|V_g^{'[0]}\right\|_{L^2(0,R_0)}.
\end{align*}
Here, the last inequality follows from the condition $p_{g,0}(0)=0$ together with the  Poincar\'e's inequality.
Therefore,
\begin{align} \label{eq:36}
    \left\|
        \widehat w^0_g(R_0,\gamma+i\xi)
    \right\|_{L^2_\xi}
    \le
    C_{\gamma,R_0}
    \left(
        \left\|V_g^{'[0]}\right\|_{L^2(0,R_0)}
        +
        \left\|V_g^{[1]}\right\|_{L^2(0,R_0)}
    \right).
\end{align}

It remains to estimate the orthogonal component. Since $a_{R_0}$, given by \eqref{eq:14}, is
nonnegative and its kernel is $\operatorname{span}\{\phi_0\}$, the form
is coercive on the orthogonal complement of this kernel. Equivalently,
there exists $c_{R_0}>0$ such that
\begin{align*}
    a_{R_0}(u,u)\ge c_{R_0}\|u\|_{H^1(0,R_0)}^2,
    \qquad
    u\in V,\quad u\perp\phi_0.
\end{align*}
Let $\{\phi_n\}_{n\ge1}$ be an $L^2(0,R_0)$-orthonormal basis of
eigenfunctions of $A_{R_0}$ on this orthogonal complement, with positive
eigenvalues $\mu_n>0$. Then
\begin{align*}
A_R\phi_n=\mu_n\phi_n,
\qquad n\ge1.
\end{align*}
Write
\begin{align*}
&  g_\perp=\sum_{n\ge1}g_n\phi_n,
    \qquad \widetilde g_\perp=\sum_{n\ge1}\widetilde g_n\phi_n,\qquad
\widehat w_g^\perp(R_0,s)=\sum_{n\ge1}\widehat w_n(s)\phi_n.
\end{align*}
It follows from \eqref{eq:42} that 
\begin{align} \label{eq:11}
    \left(s^2+c_g^2\mu_n\right)\widehat w_n(s)
    =
    s g_n + \widetilde g_n.
\end{align}
For each fixed $s$, the trace estimate and the coercivity on the
orthogonal complement imply
\begin{align*}
    \left|\widehat w_g^\perp(R_0,s)\right|^2
    \le C_{R_0} a_{R_0}\left(\widehat w_g^\perp(\cdot,s), \widehat w_g^\perp(\cdot,s)\right)=
    C_{R_0}
    \sum_{n\ge1}
    \mu_n|\widehat w_n(s)|^2.
\end{align*}
This, together with \eqref{eq:11} yields
\begin{align*}
    \left|\widehat w_g^\perp(R_0,s)\right|^2
    &\le
    {R_0}
    \sum_{n\ge1}
    \mu_n
    \frac{|s g_n + \widetilde g_n|^2}
    {|s^2+c_g^2\mu_n|^2}  \le 
    {R_0}
    \sum_{n\ge1}
    \mu_n
    \frac{|s|^2|g_n|^2+|\widetilde g_n|^2}
    {|s^2+c_g^2\mu_n|^2}.
\end{align*}
Using
\begin{align*}
    \mathcal L(\cos(c_g\sqrt{\mu}\,t))(s)
    =
    \frac{s}{s^2+c_g^2\mu},
\end{align*}
and applying the Laplace--Plancherel
identity, we obtain 
\begin{align*}
     \int_{\mathbb R}
    \frac{|s|^2}
    {|s^2+c_g^2\mu|^2}
    \,d\xi
    \le
    C_\gamma,
    \qquad
    \mu>0,
\end{align*}
Similarly, since 
\begin{align*}
    \mathcal L\left(
        \frac{\sin(c_g\sqrt{\mu}\,t)}{c_g\sqrt{\mu}}
    \right)(s)
    =
    \frac{1}{s^2+c_g^2\mu},
\end{align*}
we obtain
\begin{align*}
    \int_{\mathbb R}
    \frac{\mu}
    {|s^2+c_g^2\mu|^2}
    \,d\xi
    \le
    \frac{C_{\gamma}}{c_g^2},
    \qquad
    \mu>0.
\end{align*}
Hence
\begin{align} \label{eq:50}
    \|\widehat w_g^\perp(R_0,\gamma+i\xi)\|_{L^2_\xi(\mathbb R)}^2
    &\le
    C_{\gamma}\left(1 + \frac{1} {c^2_{g}}\right)
    \left( \sum_{n\ge1}\mu_n|g_n|^2 +
    \sum_{n\ge1}|\widetilde g_n|^2
    \right).
\end{align}
Furthermore, by the spectral representation of the form $a_{R_0}$, we have  
\begin{align*}
    &\sum_{n\ge1}\mu_n|g_n|^2
    =
    a_{R_0}(g_\perp,g_\perp)
    \le
    C_{R_0}\left\|V_g^{'[0]}\right\|_{L^2(0,R_0)}^2,\\
   & \sum_{n\ge1}|\widetilde g_n|^2
    \le
    \left\|V_g^{[1]}\right\|_{L^2(0,R_0)}^2.
\end{align*}
In conjunction with \eqref{eq:50}, we arrive at 
\begin{align}\label{eq:43}
    \|\widehat w_g^\perp(R_0\gamma+i\xi)\|_{L^2_\xi(\mathbb R)}
    \le
     C_{\gamma, R_0}\left(1 + \frac{1} {c_{g}}\right)
    \left(
        \left\|V_g^{'[0]}\right\|_{L^2(0,R_0)}
        +
       \left\|V_g^{[1]}\right\|_{L^2(0,R_0)}
    \right).
\end{align}

Combining \eqref{eq:18}, \eqref{eq:36} and \eqref{eq:43} yields \eqref{eq:31}.

\end{proof}

\section{Proofs of Lemmas \ref{le:2}, \ref{le:3} and \ref{lem:lifting-jump}} \label{sec:B}
\begin{proof}[Proof of Lemma \ref{le:2}]
Since $p_\alpha$ satisfies the homogeneous wave equation with wave speed
$c_\alpha$ in $\Omega^{(0)}_\alpha$, $\alpha \in \{\ell,g\}$, applying Green's formula in each phase, we readily obtain \eqref{eq:1-LP} and \eqref{eq:2-LP}.
\end{proof}

\begin{proof}[Proof of Lemma \ref{le:3}]
Let 
\begin{align*}
M_{\alpha}:=\operatorname{sup}_{t\in(0,T)}\left[\left(I^{\rm{in}}_\alpha h_\alpha\right)^{(m,l)}(t/c_\alpha)/{c_\alpha}\right],
\end{align*}
which is compact since $\left(I_\alpha h_\alpha\right)^{(m,l)}(t/c_\alpha)/{c_\alpha}$ is compactly
supported. Then, for each $t\in [0,T]$, we have 
\begin{align*}
&\left|\int_0^T
\left(I^{\rm{in}}_\alpha h_\alpha\right)^{(m,l)}(t)\psi(t)\,dt \right| = \left| \int^{Tc_\alpha}_0 \frac{1}{c_\alpha}
\left(I^{\rm{in}}_\alpha h_\alpha\right)^{(m,l)}(t/c_\alpha)\psi(t/c_\alpha)\,dt\right|\\
& \qquad\qquad\qquad\qquad
\qquad\le \int^{M_\alpha}_0 \left|\int_{\mathbb S^2}\frac{1}{4\pi t}
\int_{\Omega_{\alpha}\cap \{|R_0\hat x-y|= t\}} h_\alpha(y)\,dS(y) \overline{Y^m_\ell(\hat x)} dS(\hat x)\right| \left|\psi(t/c_\alpha)\right|dt\\
&\qquad \qquad\qquad\qquad
\qquad\le C_{T,h_\alpha} \|\psi\|_{L^{\infty}[0,T]}.
\end{align*}
This directly yields the desired estimate.
\end{proof}

\begin{proof}[Proof of Lemma \ref{lem:lifting-jump}]
This proof relies on the standard transmission lifting, which can be constructed from the layer-potential method. More precisely, one combines a double-layer potential, which
realizes the prescribed trace jump, with a single-layer potential whose density
is chosen by the associated two-phase boundary integral equation so that the
weighted conormal jump vanishes. The required jump relations and Sobolev
mapping properties of the single- and double-layer potentials, together with
the corresponding boundary operators, are classical; see
\cite[Theorems~6.11--6.12 and Theorem~7.1]{McLean2000}.
\end{proof}

\end{appendices}

\end{document}